\documentclass[11pt, a4paper]{article}
\usepackage{amscd, hyperref}
\usepackage[active]{srcltx}
\usepackage{amsmath,amsthm,amssymb,amsbsy,upref,color,graphicx}

\DeclareMathOperator{\argmin}{argmin}

\DeclareMathOperator{\spt}{spt}

\DeclareMathOperator{\per}{Per}

\def\ds{\displaystyle}
\def\eps{{\varepsilon}}
\def\N{\mathbb{N}}
\def\R{\mathbb{R}}
\def\A{\mathcal{A}}
\def\EE{\mathcal{E}}
\def\F{\mathcal{F}}
\def\HH{\mathcal{H}}
\def\LL{\mathcal{L}}
\def\M{\mathcal{M}}
\def\PP{\mathcal{P}}
\def\RR{\mathcal{R}}
\def\symd{{\scriptstyle\Delta}}
\newcommand{\be}{\begin{equation}}
\newcommand{\ee}{\end{equation}}
\newcommand{\bib}[4]{\bibitem{#1}{\sc#2: }{\it#3. }{#4.}}
\newcommand{\Cp}{\mathop{\rm Cap}\nolimits}
\newcommand{\cp}{\mathop{\rm cap}\nolimits}

\numberwithin{equation}{section}
\theoremstyle{plain}
\newtheorem{teor}{Theorem}[section]
\newtheorem{lemm}[teor]{Lemma}

\newtheorem{prop}[teor]{Proposition}
\newtheorem{defi}[teor]{Definition}
\theoremstyle{remark}
\newtheorem{rema}[teor]{Remark}
\newtheorem{exam}[teor]{Example}

\title{Spectral Optimization Problems}

\author{Buttazzo, Giuseppe \thanks{\scriptsize Dipartimento di Matematica,
Universit\`a di Pisa, Largo B. Pontecorvo 5, 56127 Pisa, ITALY\quad {\texttt
{buttazzo@dm.unipi.it}}}}

\begin{document}
\maketitle

\begin{abstract}
In this survey paper we present a class of shape optimization problems where the cost function involves the solution of a PDE of elliptic type in the unknown domain. In particular, we consider cost functions which depend on the spectrum of an elliptic operator and we focus on the existence of an optimal domain. The known results are presented as well as a list of still open problems. Related fields as optimal partition problems, evolution flows, Cheeger-type problems, are also considered.
\end{abstract}



\section{Introduction}\label{sintro}

Looking for optimal shapes is a very fascinating field, perhaps due to the fact that a shape is something closer to the human spirit than a function, in which a parametrization procedure is present, making it less direct and intuitive. However, from a mathematical point of view, handling shapes is much more difficult than handling functions, for instance because no vector space structure can be defined on a family of domains, which makes useless most of the functional analysis tools in Banach spaces.

A general shape optimization problem reads as
\be\label{shopt}
\min\big\{F(\Omega)\ :\ \Omega\in\A\big\}
\ee
where $\A$ is a suitable family of admissible domains and $F$ is a suitable cost function defined on $\A$. Problems of this kind arise in many fields and in many applications, and the literature is very wide, from the classical cases of isoperimetric problems and the Newton problem of the
best aerodynamical shape to the most recent applications to elasticity and to spectral optimization. Without any claim to be complete and exhaustive we quote the recent books \cite{all02,besi03,bubu05,hen06,hepi05,pir84,sozo92}, where the reader can find a lot of shape optimization problems together with all the necessary details and references.

A family of shapes is in general more rigid than a family of functions, in the sense that a sequence of shapes may have a limit which is no more a shape; then quite often problems like \eqref{shopt} do not have optimal solutions and optimal shapes do not exist, even for very natural cost functions $F$ and admissible domains $\A$. In these situations, in order to perform the analysis of minimizing sequences, a {\it relaxation} procedure is needed, enlarging the class $\A$ to objects that are only limits of shapes and defining the cost $F$ on these new objects.

It is not our goal to describe shape optimization problems and their relaxed formulations in full generality; we will limit ourselves to the cases where the cost $F$ is of the form
$$F(\Omega)=J(u_\Omega)$$
being $u_\Omega$ the solution of some elliptic PDE in $\Omega$. Two main cases fall for instance in this scheme: the case of integral functionals and the case of functions of the spectrum of an elliptic operator.

\medskip{\it Integral functionals.} Given a right-hand side $f$ we consider the PDE
$$-\Delta u=f\hbox{ in }\Omega,\qquad u\in H^1_0(\Omega)$$
which provides, for every admissible domain $\Omega\subset\R^d$, a unique solution $u_\Omega$ that we assume extended by zero outside of $\Omega$. The cost $F(\Omega)=J(u_\Omega)$ is obtained by taking
$$J(u)=\int_{\R^d}j\big(x,u(x)\big)\,dx$$
for a suitable integrand $j$.

\medskip{\it Spectral optimization.} For every admissible domain $\Omega$ consider the Dirichlet Laplacian $-\Delta$ which, under mild conditions on $\Omega$, admits a compact resolvent and so a discrete spectrum $\lambda(\Omega)$. The cost is in this case of the form
$$F(\Omega)=\Phi\big(\lambda(\Omega)\big)$$
for a suitable function $\Phi$. For instance, taking $\Phi(\lambda)=\lambda_k$ we may consider the optimization problem for the $k$-th eigenvalue of $-\Delta$:
$$\min\big\{\lambda_k(\Omega)\ :\ \Omega\in\A\big\}.$$

We will see that, adding suitable geometrical restrictions to the class $\A$ of admissible domains, gives some extra compactness properties and prevents the relaxation of minimizing sequences, thus providing in many cases the existence of optimal domains. In some other cases, the existence of optimal domains will be obtained without any geometrical restriction, as a consequence of some qualitative properties of the cost function.

The plan of the paper is as follows.
\begin{itemize}\renewcommand{\labelitemi}{$\circ$}
\item In Section \ref{sprel} we introduce all mathematical tools that are necessary to treat the optimization problems in the rest of the paper. In particular capacity, capacitary measures, $\gamma$-convergence are recalled, together with their main properties.
\item Section \ref{sinteg} deals with optimization problems for integral functionals of the form above. We show that in general one cannot expect that an optimal domain exists and the optimum has to be searched in the class of capacitary measures, that is the relaxation of the class of domains. Some necessary conditions of optimality are also provided.
\item In Section \ref{sspec} we consider the case of optimization problems for a function $\Phi\big(\lambda(\Omega)\big)$ where $\lambda(\Omega)$ is the spectrum of an elliptic operator, that we always take $-\Delta$ with Dirichlet conditions at the boundary.
\item Section \ref{stwo} deals with the particular case of cost functions which only depends on the first two eigenvalues, of the form $\Phi\big(\lambda_1(\Omega),\lambda_2(\Omega)\big)$. The particular form of the cost allows to obtain the existence of optimal domains under very mild assumptions on the function $\Phi$.
\item Section \ref{spart} is devoted to optimal partition problems, where the unknown is a partition $\bigcup_{i=1}^N\Omega_i$ of a given set $D$ in a given number $N$ of domains, and the cost function $F(\Omega_1,\dots,\Omega_N)$ is of a fairly general type. Some interesting cases of spectral costs are presented.
\item The last Section \ref{sfurth} collects some interesting directions of investigation that merit to be developed. For instance the case of Cheeger-type costs in which the product of two quantities, scaling in a different way, has to be minimized, as well as the study of {\it spectral flows}, evolution patterns that follow the gradient flow of a spectral cost functional. The case of boundary conditions different from Dirichlet type is briefly addressed and referred to the existing literature.
\end{itemize}

\section{Preliminary tools}\label{sprel}

\subsection{Capacity}\label{sscapa}

One of the key tools for treating optimization problems governed by elliptic equations is the notion of {\it capacity}.

\begin{defi}\label{dcap}
Given a subset $E$ of $\R^d$, we define the capacity of $E$ as
$$\Cp(E)=\inf\Big\{\int_{\R^d}|\nabla u|^2+u^2\,dx\ :\ u\in{\cal U}_E\Big\}\,,$$
where ${\cal U}_E$ is the set of all functions $u$ of the Sobolev space $H^1(\R^d)$ such that $u\ge1$ almost everywhere in a neighborhood of $E$. For every bounded open set $D$ of $\R^d$ we also define a local capacity as
$$\cp(E,D)=\inf\Big\{\int_D|\nabla u|^2\,dx\ :\ u\in{\cal U}_E\Big\}\,,$$
where now ${\cal U}_E$ is the set of all functions $u$ of the Sobolev space $H^1_0(D)$ such that $u\ge1$ almost everywhere in a neighborhood of $E$.
\end{defi}

Since we are mainly interested in sets with capacity zero the two notions above play the same role and in most of the situations they will be equivalent.

If a property $P(x)$ holds for all $x\in E$ except for the elements of a set $Z\subset E$ with $\Cp(Z)=0$, we say that $P(x)$ holds {\it quasi-everywhere} (shortly {\it q.e.}) on $E$, whereas the expression {\it almost everywhere} (shortly {\it a.e.}) refers, as usual, to the Lebesgue measure.

A subset $A$ of $\R^d$ is said to be {\it quasi-open} if for every $\eps>0$ there exists an open subset $A_\eps$ of $\R^d$, such that $\Cp(A_\eps{\scriptstyle\Delta}A)<\eps$, where ${\scriptstyle\Delta}$ denotes the symmetric difference of sets. Actually, in the definition above we can additionally require that $A\subset A_\eps$. Similarly, we define {\it quasi-closed} sets. The class of all quasi-open subsets of a given set $D$ will be denoted by $\A(D)$.

A function $f:D\to\R$ is said to be {\it quasi-continuous} (resp. {\it quasi-lower semicontinuous}) if for every $\eps>0$ there exists a continuous (resp. lower semicontinuous) function $f_\eps:D\to\R$ such that $\Cp(\{f\ne f_\eps\})<\eps$. It is well known (see for instance \cite{zie89}) that every function $u\in H^1(D)$ has a quasi-continuous representative $\tilde u$, which is uniquely defined up to a set of capacity zero, and given by
$$\tilde u(x)=\lim_{\eps\to0}\frac{1}{|B(x,\eps)|}\int_{B(x,\eps)}u(y)\,dy\,,$$
where $B(x,\eps)$ denotes the ball of radius $\eps$ centered at $x$. By an abuse of notation we always identify the function $u$ with its quasi-continuous representative $\tilde u$, so that a pointwise condition can be imposed on $u(x)$ for quasi-every $x\in D$. In this way, we have for every $E\subset D$
$$\cp(E,D)=\min\Big\{\int_D|\nabla u|^2\,dx\ :\ u\in H^1_0(D),\ u\ge1\hbox{ q.e. on }E\Big\}.$$
By the identification above, the quasi-open sets can be characterized as the sets of positivity of functions in $H^1(\R^d)$; more precisely, this means that $A\subset\R^d$ is quasi-open if and only if there exists a function $u\in H^1(\R^d)$ such that $A=\{\tilde u>0\}$.

For every quasi-open set $\Omega\subset D$ we denote by $H^1_0(\Omega)$ the space of all functions $u\in H^1_0(D)$ such that $u=0$ q.e. on $D\setminus \Omega$, with the Hilbert space structure inherited from $H^1_0(D)$, i.e.
$$\langle u,v\rangle_{H^1_0(\Omega)}=\langle u,v\rangle_{H^1_0(D)}.$$
Note that $H^1_0(\Omega)$ is a closed subspace of $H^1_0(D)$ (see for instance \cite{maz85,zie89}). Moreover, if $\Omega$ is an open set, 
the previous definition of $H^1_0(\Omega)$ is equivalent to the usual one
(see \cite{adhe96}).

Most of the well-known properties of Sobolev functions on open sets extend to quasi-open sets; for instance, if $\Omega_1,\Omega_2$ are two quasi-open sets disjoint in capacity, i.e. with $\Cp(\Omega_1\cap\Omega_2)=0$, then $H^1_0(\Omega_1\cup\Omega_2)=H^1_0(\Omega_1)\cap H^1_0(\Omega_2)$ in the sense that $u\in H^1_0(\Omega_1\cup\Omega_2)$ if and only if $u_{|\Omega_1}\in H^1_0(\Omega_1)$ and $u_{|\Omega_2}\in H^1_0(\Omega_2)$.

\begin{rema}\label{bbp30}
All the definitions above have a natural extension to the Sobolev spaces $W^{1,p}_0(\Omega)$ with $1<p<+\infty$. We refer to \cite{hkm93} for a review of the main definitions and properties of the $p$-capacity; in particular, when $p>d$, the $p$-capacity of a point is strictly positive and every $W^{1,p}$-function has a continuous representative and therefore, a property which holds $p$-quasi-everywhere, with $p>d$, holds in fact everywhere, which makes trivial several shape optimization problems. This is why, from the shape optimization point of view, the most interesting case is when $p\in(1,d]$.
\end{rema}

\subsection{$\Gamma$-convergence for sequences of functionals}\label{ssgamm}

In this section we recall briefly the definition and the main properties of the $\Gamma$-convergence, which revealed to be one of the key tools for handling variational problems. We do not want here to enter into the details of that theory, but only to use it in order to characterize the relaxed form of Dirichlet problems; we refer for all details to the books \cite{bra02,dm93}.

\begin{defi}\label{dgamma}
Given a sequence $(F_n)$ of functionals from a separable metric space $X$ into $\overline\R$ we say that $(F_n)$ $\Gamma$-converges to a functional $F$ if for every $x\in X$
\begin{enumerate}
\item $\forall x_n\to x\quad F(x)\le\liminf\limits_{n\to\infty}F_n(x_n)$;
\item $\exists x_n\to x\quad F(x)\ge\limsup\limits_{n\to\infty}F_n(x_n)$.
\end{enumerate}
\end{defi}

\begin{teor}\label{tgamma}
Properties of $\Gamma $-convergence.
\begin{enumerate}
\item Every $\Gamma$-limit is lower semicontinuous on $X$;
\item if $(F_n)$ is equi-coercive on $X$, that is for every $t\in\R$ the set $\bigcup_n\{F_n\le t\}$ is relatively compact in $X$, and $(F_n)$ $\Gamma$-converges to $F$, then $F$ is coercive too, and so it admits a minimum on $X$;
\item if $x_n\in\argmin F_n$ and $x_n\to x$ in $X$, then $x\in\argmin F$; \item from every sequence $(F_n)$ of functionals on $X$ it is possible to extract a subsequence $\Gamma$-converging to a functional $F$ on $X$;
\item we have $\Gamma\lim_n(G+F_n)=G+\Gamma\lim_nF_n$ for every functional $G$ which is continuous on $X$.
\end{enumerate}
\end{teor}

Here we used the notation $\argmin F$ to indicate the set of all minimum points of $F$ on $X$. The $\Gamma$-convergence provides a convergence structure on the class $\F$ of all lower semicontinuous functionals on $X$; however, in several situations the functionals under considerations fulfill an equicoercivity condition that allows us to introduce a restricted subclass:
$$\F_\Psi(X)=\{F:X\to\overline\R,\ F\hbox{ lower semicontinuous,}\ F\ge\Psi\}$$
where $\Psi$ is a given functional on $X$. The following result concerning metrizability of the $\Gamma$-convergence will be used (see \cite{dm93}, Theorem 10.22).

\begin{teor}\label{tmetrg}
If $\Psi:X\to\overline\R$ is coercive, i.e. its sublevels $\{\Psi\le t\}$ are relatively compact, then the class $\F_\Psi(X)$, endowed with the $\Gamma$-convergence structure, is a compact metric space, in the sense that there exists a compact distance $\delta_\Gamma$ on  $\F_\Psi(X)$ such that
$$\delta_\Gamma(F_n,F)\to0\quad\iff\quad\Gamma\!\lim_{n\to\infty}F_n=F.$$
\end{teor}

\subsection{$\gamma$-convergence for sequences of quasi-open sets}\label{ssgamq}

Given a bounded domain $D\subset\R^d$ and a right-hand side $f\in L^2(D)$ we are interested to study how the solution $u_\Omega$ of the elliptic problem
$$-\Delta u=f\hbox{ in }\Omega,\qquad u\in H^1_0(\Omega)$$
depends on the domain $\Omega\in\A(D)$. The solutions are extended by zero on $D\setminus\Omega$. Since the equation above can be also written as the minimum problem
$$\min\Big\{\int_D\Big[\frac{1}{2}|\nabla u|^2-f(x)u\Big]\,dx\ :\ u\in H^1_0(\Omega)\Big\},$$
by the definition of $\Gamma$-convergence introduced in Section \ref{ssgamm} the study of the limits of $(u_{\Omega_n})$ for a given sequence $(\Omega_n)$ of domains can be reduced to the study of the $\Gamma$-limit of the sequence of functionals
$$\int_D\Big[\frac{1}{2}|\nabla u|^2-f(x)u\Big]\,dx+I_{\Omega_n}(u)$$
where
$$I_\Omega(u)=\left\{
\begin{array}{ll}
0&\hbox{if }u\in H^1_0(\Omega)\\
+\infty&\hbox{otherwise.}
\end{array}
\right.$$
All the functionals above are defined on the separable metric space $H^1_0(D)$ endowed with the $L^2(D)$ convergence. Since the term $\int_D f(x)u\,dx$ is continuous, by Theorem \ref{tgamma} it is enough to study the $\Gamma$-limit of the functionals
\be\label{gammafunc}
F_n(u)=\int_D|\nabla u|^2\,dx+I_{\Omega_n}(u).
\ee
Of course, when $(\Omega_n)$ converges to $\Omega$ in a rather strong way, for instance when $\Omega_n=(Id+\eps_nV)(\Omega)$ with $\eps_n\to0$ and $V\in C^1_c(\R^d;\R^d)$, the $\Gamma$-limit of the sequence $F_n$ above is the functional
$$F(u)=\int_D|\nabla u|^2\,dx+I_{\Omega}(u).$$
However, in shape optimization problems, we are interested to work with weaker convergences, in order to ensure, via the direct methods of the calculus of variations, the existence of an optimal domain. Therefore we look for a kind of equilibrium between a convergence strong enough to provide continuity, or at least lower semicontinuity, of the cost functionals, and on the other hand weak enough to provide the compactness of minimizing sequences. When this equilibrium is achieved one obtains the existence of an optimal solution.

A very natural definition is the following one.

\begin{defi}\label{dgammao}
We say that a sequence of quasi-open sets $(\Omega_n)$ in $\A(D)$ $\gamma$-converges to a quasi-open set $\Omega\in\A(D)$ if the corresponding functionals $\int_D|\nabla u|^2\,dx+I_{\Omega_n}(u)$ $\Gamma$-converge to the functional $\int_D|\nabla u|^2\,dx+I_\Omega(u)$, in the metric space $H^1_0(D)$ endowed with the $L^2(D)$ convergence.
\end{defi}

By the properties of the $\Gamma$-convergence seen in Section \ref{ssgamm} we have $\Omega_n\to\Omega$ in the $\gamma$-convergence if and only if for every right-hand side $f\in L^2(D)$ the solutions $u_{\Omega_n,f}$ of
$$-\Delta u=f\hbox{ in }\Omega_n,\qquad u\in H^1_0(\Omega_n)$$
(extended by zero on $D\setminus\Omega_n$) converge in $L^2(D)$ to the solution $u_{\Omega,f}$ of
$$-\Delta u=f\hbox{ in }\Omega,\qquad u\in H^1_0(\Omega).$$
Note that the sequence $(u_{\Omega,f})$ is bounded in $H^1_0(D)$, hence compact in $L^2(D)$.

\begin{rema}\label{rgamma}
The following facts for the $\gamma$-convergence hold.

\begin{enumerate}
\item By its definition, and by the equicoercivity of the functionals \eqref{gammafunc}, the $\gamma$-convergence is metrizable on $\A(D)$.
\item It can be proven (see \cite{bubu05}) that the $L^2(D)$ convergence of $u_{\Omega_n,f}$ to $u_{\Omega,f}$ for every $f$ holds if and only if it holds for $f\equiv1$. Therefore a distance on $\A(D)$ equivalent to the $\gamma$-convergence is
$$d_\gamma(\Omega_1,\Omega_2)=\|u_{\Omega_1,1}-u_{\Omega_2,1}\|_{L^2(D)}.$$
\item For every $\Omega\in\A(D)$ we may define the resolvent operator $\RR_\Omega$ which associates to every $f\in L^2(D)$ the solution $u_{\Omega,f}$; in this way the $\gamma$-convergence of $(\Omega_n)$ to $\Omega$ coincides with the pointwise $L^2(D)$ convergence of resolvent operators $\RR_{\Omega_n}$ to $\RR_\Omega$.
\item It can actually be proven (see \cite{bubu05}) that if $\Omega_n\to\Omega$ in the $\gamma$-convergence, the convergence of $u_{\Omega_n,f}$ to $u_{\Omega,f}$ is in fact strong in $H^1_0(D)$ and the resolvent operators $\RR_{\Omega_n}$ converge to $\RR_\Omega$ in the operator norm $\LL\big(L^2(D)\big)$. In particular, the spectrum of $\RR_{\Omega_n}$ converges (componentwise) to the spectrum of $\RR_\Omega$, hence the spectrum of $-\Delta$ on $H^1_0(\Omega_n)$ converges (componentwise) to the spectrum of $-\Delta$ on $H^1_0(\Omega)$.
\end{enumerate}
\end{rema}

\subsection{Capacitary measures}\label{ssmeas}

Unfortunately, the $\gamma$-convergence introduced in Section \ref{ssgamq} is not compact; in other words it is possible to construct a sequence $(\Omega_n)$ of domains (even smooth) such that the corresponding solutions $u_{\Omega_n,1}$ do not converge to a function of the form $u_{\Omega,1}$.

The first example of a sequence $(\Omega_n)$ of this form was provided by Cioranescu and Murat in \cite{cimu82}: they constructed the sequence $(\Omega_n)$ by removing from $D$ a periodic array of balls of equal radius $r_n\to0$. If the radius $r_n$ is suitably chosen they proved (see also \cite{bubu05} for an interpretation in terms of $\gamma$-convergence) that the weak $H^1_0(D)$ limit of the sequence of solutions $u_{\Omega_n,1}$ satisfies the PDE
$$-\Delta u+cu=1\hbox{ in }D,\qquad u\in H^1_0(D)$$
for a suitable constant $c>0$, and thus the sequence of domains $(\Omega_n)$ cannot $\gamma$-converge to any domain $\Omega$. The sets $\Omega_n$ are represented in Figure \ref{fciomur}

\begin{figure}[h]
\centerline{\includegraphics[width=7.0cm]{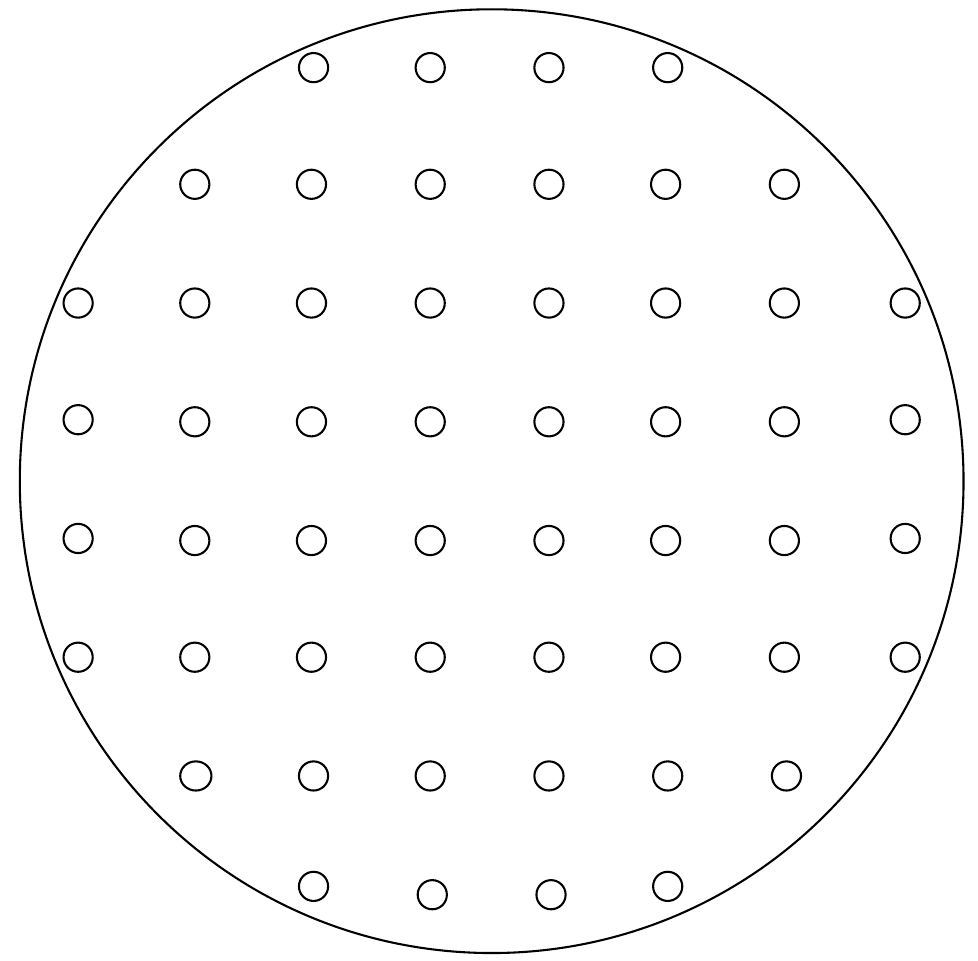}}
\caption{The sets $\Omega_n$ in the Cioranescu and Murat example.}\label{fciomur}
\end{figure}

If we indicate by $\eps_n$ the distance between two adjacent centers, then the critical size of the radius $r_n$ is
$$r_n=c\eps_n^{d/(d-2)}\hbox{ if }d>2,\qquad r_n=\exp(-c\eps_n^{-2})\hbox{ if }d=2$$
with $c>0$. In this way the capacity of the array of removed balls is of order $O(1)$, while the corresponding total volume vanishes as $\eps_n\to0$.

It is interesting to notice that, for any sequence $(\Omega_n)$ of domains in $D$ and for any $f\in L^2(D)$, the solutions $u_{\Omega_n,f}$ are bounded in $H^1_0(D)$, so up to extracting a subsequence, they weakly converge to some function $w\in H^1_0(D)$. In order to characterize this function $w$ as the solution of some limit PDE we have to identify the compactification of the metric space $\A(D)$ of all quasi-open subsets of $D$, endowed with the $\gamma$-convergence.

This identification has been obtained by Dal Maso and Mosco in \cite{dmmo87}, where it is proven that the compactification of the metric space $\A(D)$ endowed with the $\gamma$-convergence is the set $\M_0(D)$ of all nonnegative regular Borel measures $\mu$ on $D$, possibly $+\infty$ valued, such that
$$\mu(B)=0\hbox{ for every Borel set $B\subset D$ with }\Cp(B)=0.$$
We stress the fact that the measures $\mu\in\M_0(D)$ do not need to be finite, and may take the value $+\infty$ even on large parts of $D$.

\begin{exam}\label{emeas}
\begin{enumerate}
\item The Dirac measure $\delta_{x_0}$ does not belong to $\M_0(D)$ because it does not vanish on the set $\{x_0\}$ which has capacity zero when $d\ge2$. Similarly, the Hausdorff measures $\HH^k\lfloor S$, with $S$ manifold of dimension $k$, do not belong to $\M_0(D)$ when $d\ge k+2$.

\item If $d-2<\alpha\le d$ the $\alpha$-dimensional Hausdorff measure $\HH^\alpha$  belongs to $\M_0(D)$, and consequently every $\mu$ absolutely continuous with respect to $\HH^\alpha$ as well. In fact every Borel set with capacity zero has an Hausdorff dimension which is less than or equal to $d-2$.

\item For every $S\subset D$, the measure $\infty_S$ defined by
\begin{equation}\label{e421}
\infty_S(B)=\left\{
\begin{array}{ll}
0&\mbox{if }\Cp(B\cap S)=0,\\
+\infty&\mbox{otherwise}.
\end{array}
\right.
\end{equation}
belongs to the class $\M_0(D)$.
\end{enumerate}
\end{exam}

Given $\mu\in\M_0(D)$ we can consider the relaxed form of the Dirichlet problem by introducing the space $X_\mu(D)$ as the vector space of all functions $u\in H^1_0(D)$ such that $\int_D u^2\,d\mu<\infty$. Note that, since $\mu$ vanishes on all sets with capacity zero and since Sobolev functions are defined up to sets of capacity zero, the definition of $X_\mu(D)$ is well posed. We may think to $X_\mu(D)$ as to the Hilbert space (see \cite{budm91}) $H^1_0(D)\cap L^2(D,\mu)$, endowed with the norm
$$\|u\|_{X_\mu(D)}=\Big(\int_D|\nabla u|^2\,dx+\int_D u^2\,d\mu\Big)^{1/2}$$
which comes from the scalar product
$$(u,v)_{X_\mu(D)}=\int_D \nabla u\nabla v\,dx+\int_D uv\,d\mu.$$
Given $f\in L^2(D)$ we can now consider the relaxed Dirichlet problem
$$-\Delta u+\mu u=f\hbox{ in }D,\qquad u\in X_\mu(D)$$
whose precise meaning has to be given in the weak form
$$u\in X_\mu(D),\qquad\int_D\nabla u\nabla v\,dx+\int_D uv\,d\mu=\int_Dfv\,dx\quad\forall v\in X_\mu(D).$$
The usual Lax-Milgram method in the Hilbert space $X_\mu(D)$ provides, for every $\mu\in\M_0(D)$ and every $f\in L^2(D)$, a unique solution $u_{\mu,f}$. In this way we can construct the resolvent operator $\RR_\mu$ which associates to every $f\in L^2(D)$ the solution $u_{\mu,f}$.

\begin{exam}\label{x29}
\begin{enumerate}
\item Take $\mu=a(x)\,dx$ where $a\in L^p(D)$ with $p\ge d/2$ (any $p>1$ if $d=2$). Then, by the Sobolev embedding theorem and H\"older inequality, we have that $X_\mu(D)=H^1_0(D)$ with equivalent norms.

\item If $\Omega$ is a quasi-open subset of $D$ and $\mu=\infty_{D\setminus\Omega}$ then the space $X_\mu(D)$ coincides with the Sobolev space $H^1_0(\Omega)$ and the solution $u_{\mu,f}$ defined above coincides with the solution $u_{\Omega,f}$ of the Dirichlet problem in $\Omega$. Thus we can identify the domain $\Omega$ to the capacitary measure $\mu=\infty_{D\setminus\Omega}$.

\item If $\mu\in\M_0(D)$ and $f\ge0$, then by maximum principle the solution $u=\RR_\mu(f)$ is nonnegative, and then also $f+\Delta u=\mu u$ is nonnegative. On the other hand, if $f>0$ we can formally write
$$\mu=\frac{f+\Delta u}{u}$$
which gives $\mu$ once $u$ is known; of course it turns out that $\mu=+\infty$ whenever $u=0$. This argument can be made rigorous (see \cite{chdm92,bubu05}) and therefore, working with the class $\M_0(D)$ is in this case equivalent to work with the class of functions $\{u\in H^1_0(D),\ u\ge0,\ \Delta u+f\ge0\}$, which is a closed convex subset of the Sobolev space $H^1_0(D)$.
\end{enumerate}
\end{exam}

The $\gamma$-convergence can be extended to the space $\M_0(D)$: we have $\mu_n\to\mu$ in the $\gamma$-convergence if the solutions $\RR_{\mu_n}(f)$ converge to $\RR_\mu(f)$ weakly in $H^1_0(D)$ for every $f\in L^2(D)$. Again, it can be proven that this is equivalent to require the convergence only for $f\equiv1$.

\begin{prop}\label{pmeas}
The properties below for the $\gamma$-convergence on the space $\M_0(D)$ hold.

\begin{enumerate}
\item The space $\M_0(D)$ endowed with the $\gamma$-convergence is a compact metric space; a distance equivalent to the $\gamma$-convergence is
$$d_\gamma(\mu_1,\mu_2)=\|\RR_{\mu_1}(1)-\RR_{\mu_2}(1)\|_{L^2(D)}.$$

\item The class $\A(D)$ is included in $\M_0(D)$ via the identification $\Omega\mapsto\infty_{D\setminus\Omega}$ and $\A(D)$ is dense in $\M_0(D)$ for the $\gamma$-convergence. Actually also the class of all smooth domains $\Omega$ is dense in $\M_0(D)$.

\item The measures of the form $a(x)\,dx$ with $a\in L^1(D)$ belong to $\M_0(D)$ and are dense in $\M_0(D)$ for the $\gamma$-convergence. Actually also the class of measures $a(x)\,dx$ with $a$ smooth is dense in $\M_0(D)$.

\item If $\mu_n\to\mu$ for the $\gamma$-convergence, then the spectrum of the compact resolvent operator $\RR_{\mu_n}$ converges to the spectrum of $\RR_\mu$; in other words, the eigenvalues of the Schr\"odinger-like operator $-\Delta+\mu_n$ defined on $H^1_0(D)$ converge to the corresponding eigenvalues of the operator $-\Delta+\mu$.

\end{enumerate}
\end{prop}

\begin{rema}\label{rplap}
Even if in this paper we limit ourselves to consider only the case of the Laplace operator, the same construction of relaxed domains can be performed for the $p$-Laplacian $-\Delta_p$ in $W^{1,p}_0(D)$ for $1<p\le d$, and for more general classes of nonlinear operators (see \cite{dmmu97}). Given $f\in L^p(D)$ and a sequence $(\Omega_n)$ of $p$-quasi-open subsets of $D$ (i.e. defined in a similar way by means of the $p$-capacity) we denote by $u_n$ the solution of the following equation on $\Omega_n$:
$$-\Delta_p u_n=f\hbox{ in }\Omega_n,\qquad u_n\in W^{1,p}_0(\Omega_n),$$
which has to be understood in the weak sense
$$u_n\in W^{1,p}_0(\Omega_n),\qquad
\int_{\Omega_n}|\nabla u_n|^{p-2}\nabla u_n\nabla v\,dx=\int_{\Omega_n}fv\,dx\quad\forall v\in W^{1,p}_0(\Omega_n).$$
A suitable subsequence (still denoted by the same indices) of $(u_n)$ weakly converges in $W^{1,p}_0(D)$ to the solution of the equation
$$-\Delta_p u+ \mu |u|^{p-2}u=f,\qquad u\in W^{1,p}_0(D)\cap L^p_\mu(D),$$
$\mu$ being the Borel measure defined by 
$$\mu (A)=\left\{\begin{array}{ll}
\ds{\int_A\frac{d\nu}{w^{p-1}}}&\mbox{if }\Cp_p(A\cap\{w=0\})=0\\
+\infty&\mbox{if }\Cp_p(A\cap\{w=0\})>0.
\end{array}\right.$$
Here $w$ is the weak limit in $W^{1,p}_0(D)$ of the solutions above with $f=1$ and $\nu=1+\Delta_p w$.
\end{rema}

\subsection{Geometrical restrictions}\label{ssgeom}

As we have seen, the $\gamma$-convergence makes the space $\M_0(D)$ a compact metric space, while the subspace of shapes $\A(D)$ is not compact (it is even dense in $\M_0(D)$). However, if we consider smaller classes than the whole $\A(D)$, introducing geometrical constraints on the class of admissible shapes, we can provide some extra compactness properties. We refer to \cite{bubu05} for all details concerning the restricted classes below.

\bigskip$\bullet$ A very strong geometrical constraint is convexity: the class
$$\A_{convex}\hbox{ of all convex open subsets of }D$$
is compact with respect to many kinds of convergences, among which the $\gamma$-convergence.

\bigskip$\bullet$ A weaker geometrical constraint is given by the so-called {\it uniform exterior cone condition}: an open set $\Omega$ satisfies the uniform exterior cone condition if for every point $x_0\in\partial\Omega$ there exists a closed cone, with uniform height and opening, and with vertex in $x_0$, contained in the complement of $\Omega$. The class
$$\A_{unif\;cone}\hbox{ of all open subsets of $D$ with the uniform exterior cone condition}$$
is still compact for the $\gamma$-convergence.

\bigskip$\bullet$ The class $\A_{unif\;flat\;cone}$ is the class of domains satisfying a uniform flat cone condition (see \cite{buzo95}), i.e. as above, but with the weaker requirement that the cone may be flat, that is of dimension $d-1$. The class $\A_{unif\;flat\;cone}$ is larger than the previous ones and is still compact for the $\gamma$-convergence.

\bigskip$\bullet$ The previous geometrical constraints can be further weakened by considering the classes:

\medskip
$\A_{cap\;density}$ of domains $\Omega$ satisfying a uniform capacitary density condition, i.e. such that there exist $c,r>0$ with
$$\frac{\cp\big(B_t(x)\setminus\Omega,B_{2t}(x)\big)}
{\cp\big(B_t(x),B_{2t}(x)\big)}
\ge c,\qquad\forall t\in(0,r),\ \forall x\in\partial\Omega$$
where $B_s(x)$ denotes the ball of radius $s$ centered at $x$.

\medskip
$\A_{unif\;Wiener}$ of domains $\Omega$ satisfying a uniform Wiener condition, i.e. such that for every point $x\in\partial\Omega$
$$\int_r^R\frac{\cp\big(B_t(x)\setminus\Omega,B_{2t}(x)\big)}
{\cp\big(B_t(x),B_{2t}(x)\big)}\frac{dt}{t}
\ge g(r,R,x)\qquad\forall0<r<R<1$$
where $g:(0,1)\times(0,1)\times D\to\R^+$ is a fixed function such that for every $R\in(0,1)$
$$\lim_{r\to0}g(r,R,x)=+\infty\hbox{ locally uniformly  on }x.$$ 

For the classes above the following inclusions can be established:
$$\A_{convex}\subset\A_{unif\;cone}\subset\A_{unif\;flat\;cone}\subset\A_{cap\;density}\subset\A_{unif\;Wiener}$$
and on each of the classes the $\gamma$-convergence is equivalent to the Hausdorff complementary convergence $\HH^c$ (i.e. the Hausdorff convergence of the sets $D\setminus\Omega_n$) which is then compact.

\bigskip$\bullet$ Another interesting class, which is only of topological type and is not contained in any of the previous ones, was found by \v{S}ver\'ak in \cite{sve93}; it is concerned only with the case $d=2$ and consists of all open subsets $\Omega$ of $D$ for which the number of ``holes'', i.e. connected components of $\overline D\setminus\Omega$, is uniformly bounded. More precisely, the classes
$${\mathcal O}_m(D)=\big\{\Omega\subset D,\ \Omega\hbox{ open, $\#$ connected components of }\overline D\setminus\Omega\le m\big\}$$
are such that on them the $\gamma$-convergence is equivalent to the Hausdorff complementary convergence $\HH^c$, which is then compact. For higher dimensions $d\ge3$ the same result is no more true but it can be recovered again if the Laplace operator is replaced by the $p$-Laplace operator, with $p\in]d-1,d]$ (see \cite{butr98}). For $p>d$ the compactness result is trivial, due to the fact that points have a strictly positive $p$-capacity and that the Sobolev embedding theorem gives uniform continuity of Sobolev maps.


\section{Optimization problems for integral functionals}\label{sinteg}

\subsection{Nonexistence of optimal domains}\label{ssnonex}

We present an explicit example of optimization problem for an integral functional where the existence of an optimal domain does not occur. We fix a bounded open subset $D$ in $\R^d$ ($d\ge2$), a function $f\in L^2(D)$, and we consider the minimization problem
\be\label{e31}
\min\Big\{J(\Omega)=\int_Dj\big(x,u_\Omega(x)\big)\,dx\ :\ \Omega\in\A(D)\Big\}
\ee
where $u_\Omega$ is the solution of the Dirichlet problem in $\Omega$
\be\label{esteq}
-\Delta u=f\hbox{ in }\Omega,\qquad u\in H^1_0(\Omega),
\ee
extended by zero on $D\setminus\Omega$.

For instance, taking $j(x,s)=|s-u_0(x)|^2$, the problem amounts to find a state function $u$ as close as possible, in the $L^2(D)$ norm, to the desired state $u_0$, only acting on the shape of the Dirichlet region $D\setminus\Omega$, which can then be seen as a control variable, with \eqref{esteq} as the corresponding state equation. Considering the state function $u$ as the temperature of a conducting medium $D$ with prescribed heat sources $f$ in a stationary configuration, the optimization problem \eqref{e31} consists in finding an optimal distribution of the Dirichlet region $D\setminus\Omega$ in order to achieve a temperature as close as possible to the desired temperature $u_0$.

We shall see that for a large class of integrands $j(x,s)$ an optimal solution does not exist; the reason is that a Dirichlet region $D\setminus\Omega$ composed of many small pieces is more efficient, in terms of the given cost, than a domain composed by a single piece. Therefore a minimizing sequence $(\Omega_n)$ of domains is such that $D\setminus\Omega_n$ splits more and more, then leading to the nonexistence of the optimum.

We argue by contradiction and assume that the optimization problem \eqref{e31} admits a solution $\Omega$ that does not coincide with the entire set $D$; assume also for simplicity some mild regularity on $\Omega$, as that $\Omega$ is an open set whose boundary has zero Lebesgue measure and whose closure $\overline\Omega$ does not fill all the set $D$. All these extra assumptions could be actually removed (see \cite{budm90,budm91,chdm92}) only requiring that $\Omega$ is a quasi-open set. For every point $x_0\in D\setminus\overline\Omega$ we perform the so-called {\it topological derivative} on the cost functional, which consists in adding to $\Omega$ a small ball $B_\eps$ of radius $\eps$ centered at $x_0$, then obtaining a competing domain $\Omega_\eps=\Omega\cup B_\eps$. Since for small $\eps$ the ball $B_\eps$ is disjoint from $\Omega$ we obtain
$$J(\Omega_\eps)=\int_\Omega j\big(x,u_\Omega(x)\big)\,dx+\int_{B_\eps}j\big(x,u_\eps(x)\big)\,dx+\int_{D\setminus\Omega_\eps}j(x,0)\,dx$$
where we denoted by $u_\eps$ the solution of
$$-\Delta u=f\hbox{ in }B_\eps,\qquad u\in H^1_0(B_\eps).$$
The same computation for $J(\Omega)$ gives
$$J(\Omega)=\int_\Omega j\big(x,u_\Omega(x)\big)\,dx+\int_{B_\eps}j(x,0)\,dx+\int_{D\setminus\Omega_\eps}j(x,0)\,dx,$$
so that
$$J(\Omega_\eps)-J(\Omega)=\int_{B_\eps}j\big(x,u_\eps(x)\big)-j(x,0)\,dx.$$

Assuming that $j(x,s)$ is continuous in $x$ and continuously differentiable in $s$, and that the source $f$ is continuous (these conditions can be easily weakened), we obtain
$$u_\eps(x)=f(x_0)\frac{\eps^2-|x-x_0|^2}{2d}+o(\eps^2)\qquad\forall x\in B_\eps$$
so that
$$\begin{array}{lll}
j\big(x,u_\eps(x)\big)&=&j(x,0)+u_\eps(x)j_s(x,0)+o(\eps^2)\\
&=&j(x,0)+f(x_0)j_s(x,0)\big(\eps^2-|x-x_0|^2\big)/(2d)+o(\eps^2)\\
&=&\ds j(x,0)+f(x_0)j_s(x_0,0)\big(\eps^2-|x-x_0|^2\big)/(2d)+o(\eps^2).
\end{array}$$
Therefore
$$J(\Omega_\eps)-J(\Omega)=\frac{f(x_0)j_s(x_0,0)}{2d}\int_{B_\eps}\big(\eps^2-|x-x_0|^2\big)\,dx+o(\eps^{d+2})$$
and, using the optimality of $\Omega$ we end up with the necessary condition of optimality:
\be\label{e32}
f(x_0)j_s(x_0,0)\ge0\qquad\forall x_0\in D\setminus\Omega.
\ee
For instance, if $j(x,s)=|s-u_0(x)|^2$ and $f(x)>0$ on $D$, the necessary condition above becomes
$$u_0(x)\le0\qquad\hbox{on }D\setminus\Omega.$$
Taking the desired state $u_0(x)>0$ on $D$, we deduce that the only possibility for a domain $\Omega$ to be optimal is that $\Omega=D$. However, also this possibility can be excluded if the function $u_0$ is small enough; in fact, in this case the empty set gives a better value of the cost functional, since for small $u_0$ we have
$$\int_D|u_0|^2\,dx=J(\emptyset)<J(D)=\int_D|u_D-u_0|^2\,dx.$$

\subsection{Necessary conditions of optimality}\label{ssnece}

Following the argument illustrated in the previous section we can obtain the result below (see \cite{budm90,budm91} for a more detailed proof).

\begin{teor}\label{neccond}
Let $f\in L^2(D)$ and assume $j:D\times\R\to\R$ satisfies the conditions:
\begin{enumerate}
\item$j(x,s)$ is measurable in $x$ and continuous in $s$;

\item$|j(x,s)|\le a(x)+b|s|^2$ for suitable $a\in L^1(D)$ and $b\in\R$;

\item$j(x,\cdot)$ is continuously differentiable and $|j_s(x,s)|\le a_1(x)+b_1|s|$ for suitable $a_1\in L^2(\Omega)$ and $b_1\in\R$.
\end{enumerate}
Then, if a smooth domain $\Omega$ solves the shape optimization problem \eqref{e31}, then the following necessary condition of optimality holds:
\be\label{enecc}
f(x)j_s(x,0)\ge0\qquad\hbox{a.e. on }D\setminus\Omega.
\ee
\end{teor}

\begin{rema}
When the optimal set $\Omega$ is not smooth but only quasi-open, a necessary condition of optimality of the form of \eqref{enecc} has to be given in terms of the fine topology. We recall that the fine topology on $D$ is the weakest topology on $D$ for which all superharmonic functions are continuous (see for instance \cite{doob} for a detailed study of properties of the fine topology); by $\partial^*\Omega$ and $cl^*\Omega$ we denote respectively the fine boundary and the fine closure of $\Omega$ in $D$. Under the same conditions on $f$ and $j$ as in Theorem \ref{neccond} the topological derivative argument gives (see \cite{budm90,budm91}) that an optimal quasi-open set $\Omega$ must fulfill the necessary condition of optimality:
$$f(x)j_s(x,0)\ge0\quad\hbox{a.e. on }D\setminus cl^*\Omega.$$
\end{rema}

Other kinds of necessary conditions of optimality can be obtained for the optimization problem \eqref{e31}. Assuming again that a smooth optimal domain $\Omega_{opt}$ exists, it is convenient to introduce the adjoint state equation, that for any admissible $\Omega$ reads
\be\label{e435}
-\Delta v=j_s(x,u_\Omega)\hbox{ in }\Omega, \qquad v\in H^1_0(\Omega).
\ee
In \cite{budm91} the following optimality conditions have been obtained:
\be\label{eoptcond}
\left\{\begin{array}{ll}
uv\le0&\hbox{q.e. on }\Omega_{opt};\\
f(x)j_s(x,0)\ge0&\hbox{a.e. on }D\setminus\Omega_{opt};\\
(\partial u/\partial n)(\partial v/\partial n)=0
&\HH^{N-1}\hbox{-a.e. on }D\cap\partial\Omega_{opt}.
\end{array}\right.
\end{equation}

\subsection{Relaxed Dirichlet problems}\label{ssrelax}

We have seen that in general the minimization problem \eqref{e31} may have no solution and its relaxed formulation on $\M_0(D)$
\be\label{e31mu}
\min\Big\{J(\mu)=\int_Dj\big(x,u_\mu(x)\big)\,dx\ :\ \mu\in\M_0(D)\Big\},
\ee
where $u_\mu$ is the solution of the relaxed Dirichlet problem
$$-\Delta u+\mu u=f\hbox{ in }D,\qquad u\in X_\mu(D),$$
has to be considered, in order to obtain a relaxed solution, which is a measure of the class $\M_0(D)$. The previous analysis leading to necessary conditions of optimality can be performed for relaxed solutions too, considering the adjoint relaxed state equation
\be\label{e435mu}
-\Delta v+\mu v=j_s(x,u_\mu)\hbox{ in }D, \qquad v\in X_\mu(D).
\ee
Let now $\mu_{opt}$ an optimal relaxed solution in the space $\M_0(D)$ (which always exists thanks to the compactness of the $\gamma$-convergence in $\M_0(D)$). The optimality conditions \eqref{eoptcond} can be written for a general $\mu_{opt}$ as
\be\label{eoptcondmu}
\left\{\begin{array}{ll}
uv\le0&\hbox{q.e. in }D;\\
uv=0&\mu\hbox{-a.e. in }D;\\
f(x)j_s(x,0)\ge0&\hbox{a.e. on }\hbox{int}^*\{u_{\mu_{opt}}=0\},
\end{array}\right.
\end{equation}
where int* denotes the interior with respect to the fine topology. The last condition in \eqref{eoptcond}, involving normal derivatives, has a more involved translation in terms of capacitary measures, and we refer to \cite{budm91,bubu05}.

Consider now again the problem of finding the temperature closest to the desired one $u_0$, acting on the Dirichlet region $D\setminus\Omega$. We have seen that when $u_0$ is small enough and the heat source $f$ is positive, no optimal domain exists and the problem has to be written in the relaxed formulation
$$\min\Big\{\int_D|u_\mu-u_0(x)|^2\,dx\ :\ \mu\in\M_0(D)\Big\}$$
in order to obtain a solution in $\M_0(D)$. Thanks to the fact that $f>0$ and to the identification seen in Example \ref{x29} (iii), the relaxed formulation above can be rewritten as
$$\min\Big\{\int_D|u-u_0(x)|^2\,dx\ :\ u\in H^1_0(D),\ u\ge0,\ \Delta u+f\ge0\Big\}$$
which is now a strictly convex problem on the closed convex subset $\{u\in H^1_0(D),\ u\ge0,\ \Delta u+f\ge0\}$ of the Sobolev space $H^1_0(D)$. Therefore there is a unique solution $u$ from which we may deduce the optimal measure $\mu$ through the formula
$$\mu=\frac{f+\Delta u}{u}\;.$$
In particular, if $D$ is a ball and $u_0,f$ are radially symmetric, the optimal solution $u$ and so the optimal measure $\mu$ are radially symmetric too and an explicit computation can be made in polar coordinates. Let us consider the particular case when $d=2$, $D$ is the unit disk, $f\equiv1$, and $u_0\equiv c$ is a constant. After some computations we find that for small values of the constant $c$ the optimal measure $\mu$ has the form
$$\mu=\frac{1}{c}\HH^2\lfloor B_{R_c}+K_c\HH^1\lfloor\partial B_{R_c}$$
where the values of $R_c$ and of $K_c$ can be computed explicitely by one-dimensional optimization problems. Note that $u=c$ in $B_{R_c}$. A plot of the corresponding optimal temperature $u(r)$ is given below.

\begin{figure}[h]
\centerline{\includegraphics[width=10.0cm]{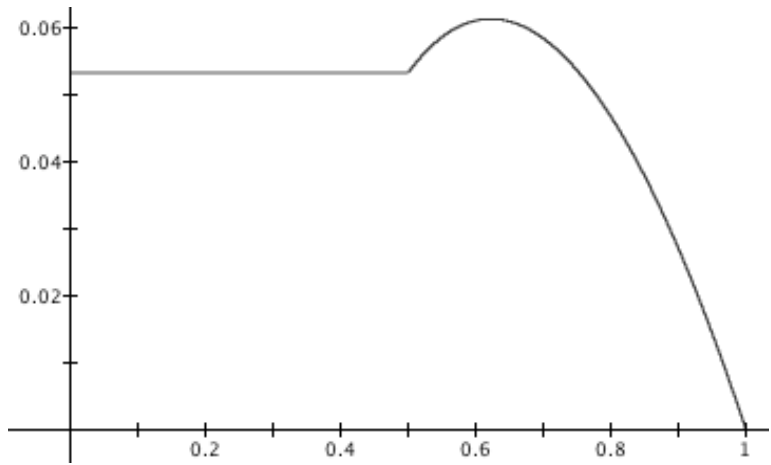}}
\caption{Plot of the optimal temperature $u(r)$ for $c$ small.}\label{fnuclear}
\end{figure}


\section{Spectral optimization problems}\label{sspec}

\subsection{A general existence result}\label{ssexis}

We have seen that in general for a shape optimization problem of the form \eqref{shopt}, either we consider on the admissible class $\A$ some additional geometrical constraints of the kind listed in Section \ref{ssgeom}, or in general the problem does not admit any optimal domain solution and the problem has to be considered in the relaxed class of capacitary measures $\M_0(D)$. However, for some particular type of cost functionals, the existence of optimal domains can be obtained (in this case we say that a {\it classical solution} exists) without the addition of any geometrical constraint but under the only volume constraint $|\Omega|\le m$. In this section we show some interesting classes of cost functionals for which the existence of classical solutions holds.

The problems we consider have the general form
\be\label{eshpb}
\min\big\{F(\Omega)\ :\ \Omega\in\A(D),\ |\Omega|\le m\big\}
\ee
where the admissible class $\A(D)$ is as before the class of quasi-open subsets of $D$. A general existence result for optimal domains has been obtained in \cite{budm93} under the following conditions on the cost functional $F$.

\begin{teor}\label{exth}
Let $F$ be a $\gamma$-lower semicontinuous functional which is decreasing with respect to the set inclusion. Then the optimization problem \eqref{eshpb} admits at least one classical solution, i.e. a domain $\Omega\in\A(D)$.
\end{teor}

\begin{rema}
Since the $\gamma$ is a rather strong convergence (for instance it implies the convergence of the spectrum, as we have seen in Proposition \ref{pmeas}) the assumption that $F$ is $\gamma$-lower semicontinuous is very mild and is always fulfilled in spectral optimization problems. On the contrary, the monotonicity assumption on $F$ is quite severe; for instance it does not hold in the case
$$F(\Omega)=\int_D|u_\Omega-u_0(x)|^2\,dx$$
considered in Section \ref{sinteg}.
\end{rema}

\begin{rema}\label{equality}
By the monotonicity assumption on the cost functional $F$ in Theorem \ref{exth}, it is clear that the inequality constraint $|\Omega|\le m$ in \eqref{eshpb} can be replaced by the equality constraint $|\Omega|=m$.
\end{rema}

\begin{exam}\label{minint}
{\it(Optimal domains for integral functionals).} Consider a given $f\in L^2(D)$, with $f\ge0$, and let $g:D\times\R\to]-\infty,+\infty]$ be a Borel function such that $g(x,\cdot)$ is lower semicontinuous and decreasing on $\R$ for a.e. $x\in D$, and $g(x,s)\ge-a(x)-bs^2$ for a suitable $a\in L^1(D)$ and $b\in\R$. For every $\Omega\in\A(D)$ let $u_\Omega=\RR_\Omega(f)$ be the unique solution of $-\Delta u=f$ on $H^1_0(\Omega)$, and let
$$F(\Omega)=\int_D g\big(x,u_\Omega(x)\big)\,dx\,.$$
Then $F$ is lower semicontinuous with respect to the $\gamma$-convergence and, since we assumed $f\ge0$, the maximum principle and the monotonicity properties of $g$ imply that $F$ is decreasing with respect to set inclusion. Therefore, by Theorem \ref{exth} (and Remark \ref{equality}) the minimum problem
$$\min\big\{\int_D g\big(x,u_\Omega(x)\big)\,dx\ :\ \Omega\in\A(D),\
|\Omega|=m\big\}$$
admits at least a solution.
\end{exam}

\begin{exam}\label{exeigen}
{\it(Optimal domains for spectral problems).} For every $\Omega\in\A(D)$ let $\lambda_k(\Omega)$ be the $k^{th}$ eigenvalue of the Dirichlet Laplacian on $H^1_0(\Omega)$, with the convention $\lambda_k(\Omega)=+\infty$ if $\Cp(A)=0$. It is well known that all the mappings $\Omega\mapsto\lambda_k(\Omega)$ are decreasing with respect to set inclusion; moreover they are continuous with respect to the $\gamma$-convergence (see Proposition \ref{pmeas} (iv) and Remark \ref{rgamma} (iv)), so that Theorem \ref{exth} (and Remark \ref{equality}) applies and for every $k\in\N$ and $0\le m\le|D|$ we obtain that the minimization problem
$$\min\big\{\lambda_k(\Omega)\ :\ \Omega\in\A(D),\ |\Omega|=m\big\}$$
admits a (classical) solution. More generally, arguing in the same way, the minimum
$$\min\big\{\Phi\big(\lambda(\Omega)\big)\ :\ \Omega\in\A(D),\ |\Omega|=m\big\}$$
is achieved too, where $\lambda(\Omega)$ denotes the spectrum of $-\Delta$ on $H^1_0(\Omega)$, that is the sequence $\big(\lambda_k(\Omega)\big)$, and the function $\Phi:[0,+\infty]^\N\to[0,+\infty]$ is lower semicontinuous and increasing, in the sense that
$$\displaylines{\lambda_k^h\to\lambda_k\quad
\forall k\in\N\ \ \Rightarrow\ \ \Phi(\lambda)
\le\displaystyle\liminf_{h\to\infty}\Phi(\lambda^h)\,,\cr
\hskip-1.25truecm\lambda_k\le\mu_k\quad
\forall k\in\N\ \ \Rightarrow\ \ \Phi(\lambda)\le\Phi(\mu)\,.\cr}$$
\end{exam}

\begin{exam}\label{excap}
{\it(Domains with minimal capacity).} Since $\Cp(E)$ is a increasing set function, the mapping $\Omega\mapsto F(\Omega)=\Cp(D\setminus\Omega)$ is decreasing with respect to the set inclusion. It is not difficult to verify that the mapping $F$ is also $\gamma$-continuous, so that the existence Theorem \ref{exth} applies, and the minimum
$$\min\big\{F(\Omega)\ :\ \Omega\in\A(D),\ |\Omega|\le m\big\}$$
is achieved. If $\F$ denotes the class of all quasi-closed subsets of $D$, passing to complements and taking $m=|D|-k$ in the previous problem, we have that the minimum
\be\label{mincap}
\min\big\{\Cp(E)\ :\ E\in\F,\ |E|=k\big\} 
\ee
is achieved. If we denote by $E_0$ a solution to \eqref{mincap} we show that
\be\label{bdm93.2.7}
\Cp(E_0)=\min\big\{\Cp(E)\ :\ E\subset D,\ |E|=k\big\}.
\ee
In fact, for every subset $E$ of $D$ there exists a quasi-closed set $E'$ such that $E\subset E'$ and $\Cp(E)=\Cp(E')$ (see for instance Section~2 of \cite{fug72}, or Proposition~1.9 of \cite{dm88}). If $|E|=k$, then $|E'|\ge k$, so that there exists $E''\in\F$ with $E''\subset E'$ and $|E''|=k$. By
\eqref{mincap}) we then have
$$\Cp(E_0)\le\Cp(E'')\le\Cp(E')=\Cp(E),$$
which proves \eqref{bdm93.2.7}.
\end{exam}

The idea of the proof of Theorem \ref{exth} is to consider a relaxed solution, that always exists, which is a measure $\mu$ of the class $\M_0(D)$. If we denote by $\overline F$ the relaxed functional, defined on $\M_0(D)$ we then have
$$\overline F(\mu)=\inf\big\{F(\Omega)\ :\ \Omega\in\A(D),\ |\Omega|\le m\big\}.$$
We associate to the measure $\mu$ the solution $u_\mu$ of the relaxed Dirichlet problem
$$-\Delta u+\mu u=1\hbox{ in }D,\qquad u\in X_\mu(D)$$
and the domain $\Omega=\{u_\mu>0\}$ which is quasi-open, since $u_\mu\in H^1_0(D)$. The new measure $\nu=\infty_{D\setminus\Omega}$ fulfills the inequality $\nu\le\mu$ and the monotonicity assumption on the cost functional gives
$$F(\Omega)=\overline F(\nu)\le\overline F(\mu),$$
then showing the optimality of the domain $\Omega$.

\subsection{The weak $\gamma$-convergence}\label{ssweakg}

To give the proof of Theorem \ref{exth} in a more rigorous way we introduce a new convergence, much weaker than $\gamma$, that makes the class $\A(D)$ compact. We call weak $\gamma$ this new convergence and we denote it by $w\gamma$.

\begin{defi}\label{dwgamma}
We say that a sequence $(\Omega_n)$ of domains in $\A(D)$ weakly $\gamma$-converges to a domain $\Omega\in\A(D)$ if the solutions $w_{\Omega_n}=\RR_{\Omega_n}(1)$ converges weakly in $H^1_0(D)$ to a function $w\in H^1_0(D)$ (that we may take quasi-continuous) such that $\Omega=\{w>0\}$.
\end{defi}

\begin{rema}
We stress the fact that, in general, the function $w$ in Definition \ref{dwgamma} does not coincide with the solution $w_\Omega=\RR_\Omega(1)$; this happens only if $\Omega_n$ $\gamma$-converges to $\Omega$, which in general does not occur, because $\gamma$-convergence is not compact on $\A(D)$. We only have that $\Omega_n$ $\gamma$ converges to some $\mu\in\M_0(D)$, so that the function $w$ in Definition \ref{dwgamma} coincides with the solution $w_\mu=\RR_\mu(1)$ of the relaxed Dirichlet problem
$$-\Delta u+\mu u=1\hbox{ in }D,\qquad u\in X_\mu(D).$$
Also, we notice that, by its definition, the $w\gamma$-convergence is compact, since the sequence $w_{\Omega_n}=\RR_{\Omega_n}(1)$ is bounded in $H^1_0(D)$ so it always has a subsequence $(\Omega_{n_k})$ weakly converging to some function $w\in H^1_0(D)$, and the set of positivity $\Omega=\{w>0\}$ (which is quasi-open since $w\in H^1_0(D)$) is then the $w\gamma$-limit of $(\Omega_{n_k})$.\\
Finally, by Definitions \ref{dgammao} and \ref{dwgamma} we obtain that the $w\gamma$-convergence is weaker than the $\gamma$-convergence.
\end{rema}

Since the $w\gamma$-convergence is rather weak, the class of $w\gamma$-lower semicontinuous functionals is much smaller than the class of $\gamma$-lower semicontinuous functionals. However, the proposition below shows that some relevant examples are still valid.

\begin{prop}\label{plebes}
Let $f\in L^1(D)$ be a nonnegative function. Then the mapping $\Omega\mapsto\int_\Omega f\,dx$ is $w\gamma$-lower semicontinuous on $\A(D)$.
\end{prop}

\begin{proof}
Let $(\Omega_n)$ be a sequence in $\A(D)$ that $w\gamma$-converges to some $\Omega\in\A(D)$; this means that $w_{\Omega_n}\to w$ in $L^2(\R^N)$ and that $\Omega=\{w>0\}$. Passing to a subsequence we may assume that $w_{\Omega_n}\to w$ a.e. on $D$. Suppose $x\in\Omega$ is a point where $w_{\Omega_n}(x)\to w(x)$. Then $w(x)>0$, and for $n$ large enough we have that $w_{\Omega_n}(x)>0$. Hence $x\in\Omega_n$. So we have shown that
$$1_\Omega(x)\le\liminf_{n\to+\infty}1_{\Omega_n}(x)\qquad\hbox{for {\it a.e.} }x\in D.$$
Fatou's lemma now completes the proof.
\end{proof}

In order to show the $w\gamma$-lower semicontinuity of other shape functionals, the following lemma is needed.

\begin{lemm}\label{lgm1}
Let $(\Omega_n)$ be a sequence of quasi-open sets $w\gamma$-converging to a quasi-open set $\Omega$. Then there exists a subsequence (still denoted by the same indices) and a sequence of quasi-open sets $G_n\subset D$ with $\Omega_n\subset G_n$ such that $G_n$ $\gamma$-converges to $\Omega$.
\end{lemm}

\begin{proof}
Let us denote by $w_\Omega$ the solution $\RR_\Omega(1)$ of the Dirichlet problem
$$-\Delta u=1\hbox{ in }\Omega,\qquad u\in H^1_0(\Omega);$$
we have $w=w_\Omega=0$ on $D\setminus\Omega$ and (see for instance \cite{dmmu97} for a detailed proof) $w\le w_\Omega$ on $\Omega$, and hence $w\le w_\Omega$ on the entire set $D$.\\
For each $\eps>0$ we define the quasi-open set $\Omega^\eps=\{w_\Omega>\eps\}$. For a subsequence, we may suppose that
$$w_{\Omega_n\cup\Omega^\eps}\hbox{ converge to some $w^\eps$ weakly in }H^1_0(D)$$
and, since by a comparison principle we have $w_{\Omega_n\cup\Omega^\eps}\ge w_{\Omega^\eps}$, passing to the limit as $n\to+\infty$ we have that $w^\eps\ge w_{\Omega^\eps}$. Let us show that $w^\eps\in H^1_0(\Omega)$. Indeed, defining $v^\eps=1-\frac{1}{\eps}\min\{w_\Omega,\eps\}$ we get $0\le v^\eps\le1$, $v^\eps=0$ on $\Omega^\eps$, and $v^\eps=1$ on $D\setminus\Omega$.\\
Taking $u_n=\min\{v^\eps,w_{\Omega_n\cup\Omega^\eps}\}$ we get $u_n=0$ on $\Omega^\eps\cup(D\setminus(\Omega_n\cup\Omega^\eps))$, and in particular on $D\setminus\Omega_n$. Moreover $u_n$ converges to $\min\{v^\eps,w^\eps\}$ weakly in $H^1_0(D)$ and hence $\min\{v^\eps,w^\eps\}$ vanishes q.e. on $\{w=0\}$. Since $v^\eps=1$ on $D\setminus\Omega$ we get that $w^\eps=0$ q.e. on $D\setminus\Omega$. Using Theorem 5.1 of \cite{dmmu97}, from the fact that $-\Delta w_{\Omega_n\cup\Omega^\eps}\le1$ in $D$ we get $-\Delta w^\eps\le1$ and hence $w^\eps\le w_\Omega$. Finally $w_{\Omega^\eps}\le w^\eps\le w_\Omega$, and by a diagonal extraction procedure we get that
$$w_{\Omega_n\cup\Omega^{\eps_n}}\hbox{ converge to $w_\Omega$ weakly in }H^1_0(D).$$
Therefore the quasi-open sets $G_n=\Omega_n\cup\Omega^{\eps_n}$ $\gamma$-converge to $\Omega$, which concludes the proof.
\end{proof}

For monotone decreasing shape functionals $F$ the $\gamma$-lower semicontinuity and the $w\gamma$-lower semicontinuity coincide, as the following result shows.

\begin{prop}\label{pmonot}
Let $F:\A(D)\to[-\infty,+\infty]$ be a $\gamma$-lower semicontinuous shape functional which is monotone decreasing with respect to the set inclusion. Then $F$ is $w\gamma$-lower semicontinuous.
\end{prop}

\begin{proof}
Let us consider a sequence of quasi-open sets $(\Omega)_n$ in $\A(D)$ $w\gamma$-converging to some quasi-open set $\Omega\in\A(D)$, and let $\big\{\Omega_{n_k}\big\}$ be a subsequence such that
$$\lim_{k\to+\infty}F(\Omega_{n_k})=\liminf_{n\to+\infty}F(\Omega_n).$$
By Lemma \ref{lgm1} there exists a subsequence (which we still denote by
$\{\Omega_{n_k}\}$) and $G_{n_k}\supset\Omega_{n_k}$ such that $G_{n_k}$ $\gamma$-converge to $\Omega$. The $\gamma$-lower semicontinuity of $F$ gives
$$F(\Omega)\le\liminf_{k\to+\infty}F(G_{n_k})$$
and the monotonicity of $F$ gives
$$F(G_{n_k})\le F(\Omega_{n_k}).$$
Therefore
$$F(\Omega)\le\liminf_{k\to+\infty}F(G_{n_k})\le\liminf_{k\to+\infty}F(\Omega_{n_k})=\liminf_{n\to+\infty}F(\Omega_n)$$
which proves the required $w\gamma$-lower semicontinuity.
\end{proof}

\proof[Proof of Theorem \ref{exth}] The proof of Theorem \ref{exth} is now straightforward. In fact, if $(\Omega_n)$ is a minimizing sequence of quasi-open sets for the optimization problem \eqref{eshpb}, by the compactness of the $w\gamma$-convergence we may extract a subsequence (still denoted by the same indices) that $w\gamma$-converges to some quasi-open set $\Omega\in\A(D)$. By Proposition \ref{plebes} we have
$$|\Omega|\le\liminf_{n\to+\infty}|\Omega_n|\le m$$
and by Proposition \ref{pmonot} we have
$$F(\Omega)\le\liminf_{n\to+\infty}F(\Omega_n).$$
Therefore $\Omega$ is an optimal set for the minimum problem \eqref{eshpb}.
\endproof

\begin{exam}\label{elambda1}
The minimization
$$\min\big\{\lambda_1(\Omega)\ :\ \Omega\in\A(D),\ |\Omega|\le m\big\}$$
is one of the first examples of spectral optimization problems. By Theorem \ref{exth} we know that an optimal solution exists and, it was conjectured by lord Rayleigh that, for $D$ large enough (to contain at least a ball of measure $m$) the optimal domain is a ball of measure $m$. The first proof of this fact was obtained by Faber and Krahn by symmetrization techniques (see \cite{fab23,kra24,kra26}). On the contrary, if $D$ does not contain a ball of measure $m$, the optimal sets $\Omega_{opt}$ have to touch the boundary $\partial D$; it has been shown (see \cite{hay00,brla10}) that such sets $\Omega_{opt}$ are actually open sets and that their free boundary (i.e. the part of $\partial\Omega_{opt}$ included in $D$) is smooth. However, the free boundary $D\cap\partial\Omega_{opt}$ does not contain any part of spherical surface (see \cite{heou01,heou03} and also \cite{hen06} Theorem 3.4.1), in the sense that no part of $D\cap\partial\Omega_{opt}$ locally coincides with a sphere.\\
When the bounding box $D$ is convex we expect that the optimal sets $\Omega_{opt}$ for $\lambda_1$ are convex too; however, this result, even if very natural and strongly expected, is not yet available, and the question is still open.
\end{exam}

\begin{exam}\label{elambda2}
The minimization problem
$$\min\big\{\lambda_2(\Omega)\ :\ \Omega\in\A(D),\ |\Omega|\le m\big\}$$
also verifies the assumptions of the existence Theorem \ref{exth}; as proved in \cite{kra26,pol55}, for $D$ large enough (to contain at least two disjoint balls of measure $m/2$ each) the optimal domain is the union of two disjoint balls of measure $m/2$ each. As before, if $D$ is not large enough in the sense above, the optimal sets $\Omega_{opt}$ have to touch the boundary $\partial D$, but in this case, even if it seems reasonable to conjecture that the free boundary of optimal sets is regular, the regularity question is still open. Actually, the proof that $\Omega_{opt}$ are {\it open} sets is not yet available.\\
Minimizing $\lambda_2(\Omega)$ in the more restricted class $\big\{\Omega\in\A(D),\ |\Omega|\le m,\ \Omega\hbox{ convex}\big\}$ also admits an optimal solution $\Omega_{opt}$, as it is easy to prove. When $D$ is large enough, since without the convexity assumption the solution is given by two equal disjoint balls, a reasonable expectation, also supported in \cite{tro73} by some numerical computations, is that $\Omega_{opt}$ is a {\it stadium}, i.e. the convex hull of two equal disjoint balls tangent each other. This conjecture has been disproved in \cite{heou01,heou03} where again it has been shown that, even if $\Omega_{opt}$ is very close to a stadium, the nonflat parts do not locally coincide with a sphere.
\end{exam}

When $k\ge3$ the optimal shapes for the minimization problem
$$\min\big\{\lambda_k(\Omega)\ :\ \Omega\in\A(D),\ |\Omega|\le m\big\},$$
that exist again thanks to the existence Theorem \ref{exth}, are not known. For $k=3$ and $D$ large enough the conjecture, still open, is that:

\medskip{\it the optimal domains for $\lambda_3$ are balls if the dimension $d$ is $2$ or $3$ and the union of three equal disjoint balls when $d\ge4$.}\medskip

\noindent For $k=4$, $d=2$ and $D$ large enough the conjecture, also still open is:

\medskip{\it the optimal domains for $\lambda_4$ in dimension $d=2$ are the union of two disjoint balls of different radius, whose radii are in the ratio $(j_{0,1}/j_{1,1})^{1/2}\sim0.79$, where $j_{0,1}$ and $j_{1,1}$ are the first zeroes of the Bessel functions $J_0$ and $J_1$ respectively.}\medskip

\noindent For $k\ge5$ one could expect that optimal sets are also made by a suitable array of balls. This is false, as shown by Keller and Wolf for $d=2$ and $k=13$; in fact in \cite{oud04} numerical computation are provided for $d=2$ showing that for $k\ge5$ optimal arrays of disks are not optimal (see Figure \ref{flambdak}).

\begin{figure}[h]
\centerline{\includegraphics[width=12.0cm]{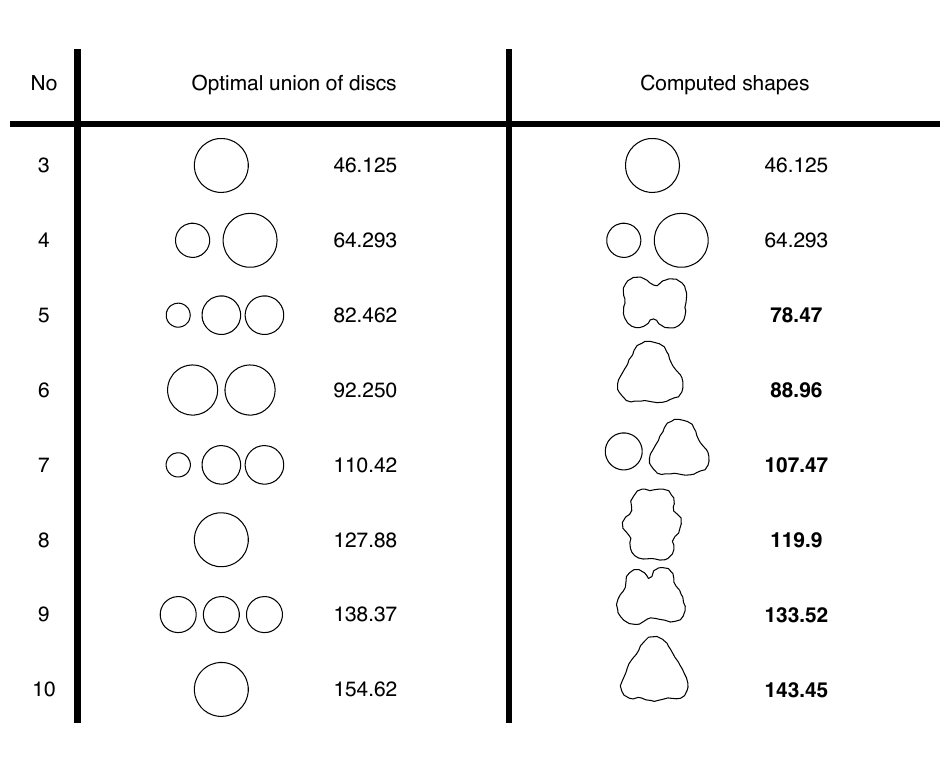}}
\caption{Optimal domains for $\lambda_k$ and optimall arrays of disks, $3\le k\le10$.}\label{flambdak}
\end{figure}

\begin{rema}\label{unbounded}
In Theorem \ref{exth} the assumption that admissible domains are all contained in a given {\it bounded} domain $D$ is crucial. A similar existence result for $D=\R^d$ is not known. For instance, the eigenvalue optimization problem
$$(\PP_k)\hskip2truecm\min\big\{\lambda_k(\Omega)\ :\ \Omega\subset\R^d,\ |\Omega|\le m\big\}$$
is known to have a solution for $k=1$ (a ball of measure $m$), for $k=2$ (two disjoint balls of measure $m/2$ each) and for $k=3$ (a ball conjectured if $d=2,3$ and three equal disjoint balls for $d\ge4$). This last existence result was proved in \cite{buhe00}, where more generally it is shown that if the minimizers of problem $(\PP_j)$ exist and are {\it bounded} for $j=1,\dots,k-1$, then the minimizers of problem $(\PP_k)$ exist. For instance, a proof of boundedness of minimizers of $\lambda_3(\Omega)$ with $\Omega\subset\R^d$ would imply the existence of optimal domains for $\lambda_4(\Omega)$ with $\Omega\subset\R^d$.
\end{rema}

Several other explicit solutions for minimization problems involving eigenvalues are known or conjectured; we refer the interested reader to \cite{ash99,bubu05,hen03,hen06,hepi05} for a wider presentation of this subject.

\subsection{Spectral problems with perimeter constraint}\label{ssperim}

We consider in this section the case when a perimeter constraint is imposed on the admissible domains. Here we use the De Giorgi definition of perimeter
(see for instance \cite{afp01}): a measurable set $\Omega$ is said to have a {\it finite perimeter} if its characteristic function
$$1_\Omega(x)=\left\{\begin{array}{ll}
1&\hbox{ if }x\in\Omega\\
0&\hbox{ otherwise}\end{array}\right.$$
belongs to the space $BV(\R^d)$ of functions with bounded variation on $\R^d$, i.e. functions in $L^1(\R^d)$ whose distributional gradient is a measure with finite total variation. In this case the perimeter of $\Omega$ is defined by
$$\per(\Omega)=\int_{\R^d}|\nabla1_\Omega|.$$
Since for a sequence of domains equi-bounded perimeter implies strong $L^1$ compactness of the characteristic functions, the situation with perimeter constraint on the admissible class could seem a priori better than the one with volume constraint. However the following difficulties arise.

\begin{enumerate}
\item[$\bullet$] The relaxation of Dirichlet problems, and hence the appearance of capacitary measures as limits of minimizing sequences, may occur even with a perimeter constraint on the admissible domains. It is enough to look at the Cioranescu and Murat example presented in Section \ref{ssmeas}: the domains $\gamma$-converging to the Lebesgue measure can be taken with perimeter arbitrarily small.

\item[$\bullet$] The perimeter is not $w\gamma$-lower semicontinuous. Indeed, again by a construction like the one of Cioranescu and Murat, we may produce a sequence $(\Omega_n)$ of subsets of the ball $B_2$ centered at the origin and of radius 2 made by removing $n$ small holes from the ball $B_1$ of radius 1. If $r_n$ denotes the critical radius in Cioranescu and Murat example, it is enough to take the radius $s_n$ of the small holes such that
$$r_n\ll s_n\ll n^{1/(1-d)}.$$
In this way, the limit set is $\Omega=B_2\setminus B_1$ and
$$\per(B_2)+\per(B_1)=\per(\Omega)>\lim_{n\to+\infty}\per(\Omega_n)=\per(B_2).$$

\item[$\bullet$] The perimeter constraint gives the compactness of minimizing sequences in the $L^1$ convergence of characteristic functions, but the $\gamma$-limit can be considerably smaller. Indeed, the same example as above, with removed balls radius $s_n$ such that
$$r_n\ll s_n\ll n^{-1/d}$$
shows that $\Omega_n\to B_2$ in the $L^1$ sense, while $\Omega_n\to B_2\setminus B_1$ in the $\gamma$-convergence sense.
\end{enumerate}

The link between $w\gamma$-convergence and $L^1$-convergence is given by the following proposition.

\begin{prop}\label{link}
Let $(A_n)$ be a sequence of quasi-open sets which $w\gamma$-converges to a quasi-open set $A$, and assume that there exist measurable sets $\Omega_n$ such that $A_n\subset\Omega_n$, and that $(\Omega_n)$ converges in $L^1$ to a measurable set $\Omega$. Then we have $|A\setminus\Omega|=0$.
\end{prop}

\begin{proof}
By applying Proposition \ref{plebes} with $f=1_{D\setminus\Omega}$ we obtain
$$|A\setminus\Omega|\le\liminf_{n\to\infty}|A_n\setminus\Omega|=
\liminf_{n\to\infty}|A_n\setminus\Omega_n|=0,$$
which concludes the proof.
\end{proof}

In order to consider shape optimization problems with perimeter constraints it is convenient to extend the definition of a monotone decreasing functional $F$ defined on $\A(D)$ also to measurable sets by setting
\begin{equation}\label{exten}
F(A)=\inf\big\{F(\Omega)\ :\ \Omega\subset A\,a.e.,\ \Omega\in\A(D)\big\}.
\end{equation}
We notice that, since $F$ is monotone decreasing, we have that its value on $\A(D)$ is not modified by this extension.

The shape optimization problems we are interested in are of the form
\begin{equation}\label{eshper}
\inf\big\{F(\Omega)\ :\ \Omega\subset D,\ \per(\Omega)\le L\big\},
\end{equation}
where $L$ is a given positive real number. 

\begin{teor}\label{texper}
If $F$ is $\gamma$-lower semicontinuous and decreasing with respect to the set inclusion, then there exists a finite perimeter set $\Omega_{opt}$ which solves the variational problem \eqref{eshper}.
\end{teor}

\begin{proof}
Let $(\Omega_n)$ be a minimizing sequence for problem \eqref{eshper}. Since $\per(\Omega_n)\le L$ we may extract a subsequence (still denoted by $(\Omega_n)$) that converges in $L^1$ to a set $\Omega$ with $\per(\Omega)\le L$.

By the construction of the extension of the functional $F$ to measurable sets there are quasi-open sets $\omega_n\subset\Omega_n$ a.e. such that
$$F(\omega_n)=F(\Omega_n).$$
By the compactness of $w\gamma$-convergence we may assume that $(\omega_n)$ is $w\gamma$-converging to some quasi-open set $\omega$, and by Proposition \ref{link} we have $|\omega\setminus\Omega|=0$. Therefore, we have that
$$F(\Omega)\le F(\omega)\le\liminf_{n\to\infty}F(\omega_n)=\liminf_{n\to\infty}F(\Omega_n).$$
Hence $\Omega$ solves the variational problem \eqref{eshper}.
\end{proof}

Of course, if $F(\Omega)=\lambda_1(\Omega)$ and $D$ is large enough to contain a ball of perimeter $L$, the balls solve the shape optimization problem \eqref{eshper} and are the unique minimizers. Indeed, symmetrizing $\Omega$ both reduces the cost $\lambda_1(\Omega)$ as well as the perimeter $\per(\Omega)$.

The situation is different for the cost $F(\Omega)=\lambda_2(\Omega)$, which was first considered in \cite{gasc53}. We summarize here the results obtained in \cite{bbh09}, to which we refer for the detailed proofs.

\begin{enumerate}
\item[$\bullet$] In the case $d=2$ it is easy to see that the optimal set $\Omega_{opt}$ has to be convex, since convexification both reduces the cost $\lambda_2$ as well as the perimeter.

\item[$\bullet$] In the case $d=2$, if $\Omega_{opt}$ is a minimizer for $\lambda_2$ with perimeter constraint, then $\lambda_2(\Omega)$ is simple.

\item[$\bullet$] Using the facts above, in the case $d=2$ it is possible to prove the $C^\infty$ regularity of $\Omega_{opt}$. Moreover, if $u_2$ denotes the second eigenfunction, with unitary $L^2$ norm, we have the necessary condition of optimality
$$|\nabla u_2(x)|^2=\frac{2\lambda_2(\Omega_{opt})}{\per(\Omega_{opt})}k(x)\qquad\forall x\in\partial\Omega_{opt}$$
where $k(x)$ is the curvature of $\partial\Omega_{opt}$ at $x$.

\item[$\bullet$] Always in the case $d=2$ additional necessary conditions of optimality can be proved, as:\\
- the boundary of $\Omega_{opt}$ does not contain any segment;\\
- the boundary of $\Omega_{opt}$ does not contain any arc of circle;\\
- the boundary of $\Omega_{opt}$ contains exactly two points where the curvature vanishes.\\
A numerical plot of $\Omega_{opt}$ in the two-dimensional case is given in Figure \ref{fperimeter}. The picture shows that minimizers in two dimensions 
should have two axes of symmetry (one of these containing the nodal line), 
but the proof of this fact is not yet available.

\item[$\bullet$] If $d\ge3$ no regularity results for the optimal domains are available; actually, at present we do not even know if optimal domains are open sets. For a similar problem with perimeter penalization the regularity of optimal domains has been proved in \cite{acks01}.

\item[$\bullet$] If $d\ge3$ optimal domains are not convex (see \cite{ivdb10}); they should have a cylindrical symmetry, even if this fact has not yet been proved.

\end{enumerate}

\begin{figure}[h]
\centerline{\includegraphics[width=11.0cm]{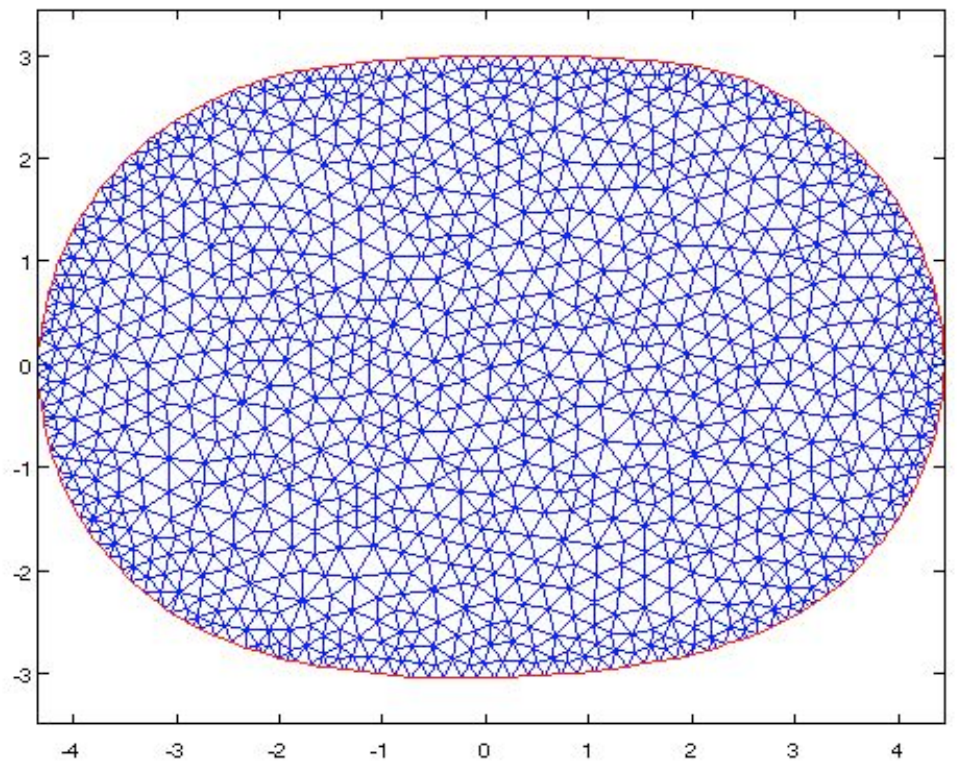}}
\caption{Plot of the optimal set for $\lambda_2$ with perimeter constraint, in the case $d=2$.}\label{fperimeter}
\end{figure}


\section{An example of nonmonotone cost functional}\label{stwo}

In this section we consider as a particular cost functional the case of a function which only depends on the first two eigenvalues of the Dirichlet Laplacian:
$$F(\Omega)=\Phi\big(\lambda_1(\Omega),\lambda_2(\Omega)\big).$$
By the results of Section \ref{ssexis} (see Theorem \ref{exth}) we know that the optimization problem
\be\label{etwo}
\min\Big\{\Phi\big(\lambda_1(\Omega),\lambda_2(\Omega)\big)\ :
\ \Omega\in\A(D),\ |\Omega|\le m\Big\},
\ee
without additional geometrical constraints on the admissible domains $\Omega$, admits a classical solution whenever the function $\Phi$ is increasing in each of its variables. The following very natural question then arises:
\medskip\\
\centerline{\it what happens when $\Phi$ does not satisfy the monotonicity condition above?}\medskip

In the rest of this section we suppose the bounding box $D$ is large enough to allow all the constructions that will be made. It is convenient to introduce the {\it attainable set}
$$\EE=\big\{(x,y)\in\R^2\ :\ x=\lambda_1(\Omega),\ y=\lambda_2(\Omega)\hbox{ for some }\Omega\in\A(D),\ |\Omega|\le m\big\},$$
so that the optimization problem \eqref{etwo} can be rewritten as
\be\label{ertwo}
\min\big\{\Phi(x,y)\ :\ (x,y)\in\EE\big\}.
\ee
The set $\EE$ is a subset of $\R^2$ which verifies the following properties.
\begin{enumerate}
\item Being the coordinates of points in $\EE$ the first and second eigenvalues of the Dirichlet Laplacian in some domain $\Omega$, we have that $x>0$ and $y>0$ for every $(x,y)\in\EE$.
\item Since $\lambda_2(\Omega)\ge\lambda_1(\Omega)$ we have that $x\le y$ for every $(x,y)\in\EE$.
\item Since the ball $B$ of measure $m$ makes $\lambda_1(\Omega)$ minimal, we have $x\ge\lambda_1(B)$ for every $(x,y)\in\EE$.
\item Since the union $A$ of two disjoint balls of measure $m/2$ each makes $\lambda_2(\Omega)$ minimal, we have $y\ge\lambda_2(A)$ for every $(x,y)\in\EE$.
\item Taking the domain $\Omega/t$ with $t\ge1$ and using the fact that $\lambda_k(\Omega/t)=t^2\lambda_k(\Omega)$, we have that the set $\EE$ is {\it conical}, i.e. if $(x,y)\in\EE$ then $(tx,ty)\in\EE$ for all $t\ge1$.
\item By a result obtained in \cite{asbe91} (proving a conjecture by Payne, P\'olya and Weinberger stated in \cite{ppw56}) the balls minimize the ratio $\lambda_1(\Omega)/\lambda_2(\Omega)$ among all domains $\Omega$, therefore $y\le x\lambda_2(B)/\lambda_1(B)$ for every $(x,y)\in\EE$.
\end{enumerate}
A numerical output of the set $\EE$ in the two-dimensional case has been obtained in \cite{woke94} and is reported below.
\begin{figure}[h]
\centerline{\includegraphics[width=9.0cm]{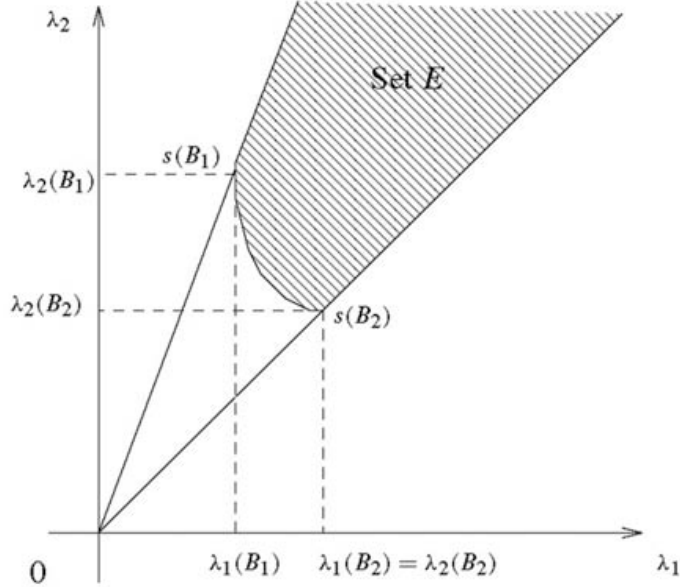}}
\caption{A numerical plot of the set $\EE$.}\label{fwoke}
\end{figure}

If the function $\Phi$ is lower semicontinuous on $\R^2$ and satisfies the coercivity condition
\be\label{ecoerc}
\lim_{|(x,y)|\to+\infty}\Phi(x,y)=+\infty
\ee
then the existence of a solution to problem \eqref{ertwo}, and then to problem \eqref{etwo}, follows straightforward by the Weierstrass theorem as soon as we can prove that the set $\EE$ is closed in $\R^2$.

\begin{prop}\label{pconvex}
If the set $\EE$ is convex, then it has to be closed.
\end{prop}

\begin{proof}
Let $(\bar x,\bar y)\in\overline\EE$ and consider the minimization problem
$$\min\big\{(x-\bar x)^++(y-\bar y)^+\ :\ (x,y)\in\EE\big\},$$
which is equivalent to
$$\min\big\{\big(\lambda_1(\Omega)-\bar x\big)^++\big(\lambda_2(\Omega)-\bar y\big)^+\ :\ \Omega\in\A(D),\ |\Omega|\le m\big\}.$$
Since the function $(x-\bar x)^++(y-\bar y)^+$ is increasing in each of its variables, by the existence Theorem \ref{exth} the minimum problems above admit optimal solutions $(x_{opt},y_{opt})\in\EE$ and $\Omega_{opt}\in\A(D)$ respectively. Since $(\bar x,\bar y)\in\overline\EE$ it is also clear that the minimum value has to be zero, so that
$$x_{opt}\le\bar x,\qquad y_{opt}\le\bar y.$$
If we are assuming that the set $\EE$ is convex, the segment joining $(x_{opt},y_{opt})$ to the point of $\EE$ given by $\big(\lambda_1(B),\lambda_2(B)\big)$, where $B$ is the ball of measure $m$, is all contained in $\EE$, as well as the segment joining $(x_{opt},y_{opt})$ to the point of $\EE$ given by $\big(\lambda_1(\tilde B),\lambda_2(\tilde B)\big)$, where $\tilde B$ is the union of two disjoint balls of measure $m/2$ each. By the conicity property (v) above of the set $\EE$ the point $(\bar x,\bar y)$ has to belong to $\EE$ thus proving that $\EE$ is closed.
\end{proof}

Unfortunately, in spite of the numerical evidence of the convexity of $\EE$ provided by Figure \ref{fwoke}, a proof of the convexity of $\EE$ is still missing, so the result of Proposition \ref{pconvex} above is useless to deduce the closedness of $\EE$ and then the existence of optimal domains for the minimization problem \eqref{etwo}. However, a result stronger than the one of Proposition \ref{pconvex} holds.

\begin{defi}\label{dhorver}
We say that the set $\EE$ is:
\begin{itemize}
\item horizontally convex if for every point $(x_0,y_0)\in\EE$ the horizontal segment joining $(x_0,y_0)$ to the straight line $y=x$ is all contained in $\EE$;
\item vertically convex if for every point $(x_0,y_0)\in\EE$ the vertical segment joining $(x_0,y_0)$ to the straight line $y=kx$, with $k=\lambda_2(B)/\lambda_1(B)$ being $B$ any ball of measure $m$, is all contained in $\EE$.
\end{itemize}
\end{defi}

Of course, if $\EE$ is convex it is also horizontally convex and vertically convex, while the converse could not in principle be true. It happens that horizontal and vertical convexity are still enough to prove that $\EE$ is closed.

\begin{prop}\label{phorver}
If the set $\EE$ is both horizontally convex and vertically convex, then it has to be closed.
\end{prop}

\begin{proof}
It is enough to repeat the proof of Proposition \ref{pconvex} and to conclude that $(\bar x,\bar y)\in\EE$ by using horizontal and vertical convexity instead of convexity.
\end{proof}

While the convexity of the set $\EE$ is still unproved, the horizontal and vertical convexity can be showed; we give a sketch of the proof omitting some technical details. A complete proof can be found in \cite{bubufi99}.

\proof[Proof of horizontal convexity]
We have to show that, given a domain $\Omega\in\A(D)$ with $|\Omega|\le m$, we can construct domains $\Omega(t)\in\A(D)$ with $|\Omega(t)|\le m$ such that
\be\label{ehorcon}
\lambda_2(\Omega(t))=\lambda_2(\Omega),\qquad\lambda_1(\Omega(t))=t\lambda_2(\Omega)+(1-t)\lambda_1(\Omega),\qquad\forall t\in[0,1].
\ee
Assume that $\Omega$ is a regular set; if we denote by $u$ a second eigenfunction for $\Omega$ and by $S$ the corresponding {\it nodal set}, i.e. the set $\{u=0\}$, we may continuously shrink $S$, obtaining then for $t\in[0,1]$ sets $S(t)\subset S$ continuously shrinking from $S$ to $\emptyset$. The sets $\Omega(t)=\Omega\setminus S(1-t)$ then $\gamma$-continuously move from $\Omega$ to $\Omega\setminus S$. Since $S$ is a nodal set we have
$$\lambda_2\big(\Omega(t)\big)=\lambda_2(\Omega)\qquad\forall t\in[0,1].$$
On the other hand, since $\Omega(t)$ $\gamma$-continuously decreases from $\Omega$ to $\Omega\setminus S$, we have that $\lambda_1\big(\Omega(t)\big)$ continuously increases from $\lambda_1(\Omega)$ to $\lambda_1(\Omega\setminus S)=\lambda_2(\Omega)$.\\
A reparametrization of the map $t\mapsto\Omega(t)$ now gives the required property \eqref{ehorcon}.
\endproof

A more careful analysis is made in \cite{bubufi99} where the following result is shown. The application to horizontal convexity is then as above, with $\Omega_0=\Omega$ and $\Omega_1=\Omega\setminus S$.

\begin{prop}\label{phomotopy}
Let $\Omega_1\subset\Omega_0$ be two quasi-open sets. Then there exists a decreasing homotopy from $\Omega_0$ to $\Omega_1$ which is $\gamma$-continuous, namely there exists a $\gamma$-continuous mapping $h:[0,1]\to\A(D)$ such that
$$h(t)\subset h(s)\hbox{ for }s<t,\quad h(0)=\Omega_0,\quad h(1)=\Omega_1.$$
\end{prop}

An important tool that we shall use in the proof of horizontal convexity is the so-called continuous Steiner symmetrization (see \cite{bro95, bro00}). Roughly speaking it consists in a path $t\mapsto\Omega_t$ starting at $t=0$ from any domain $\Omega\in\A(D)$ and symmetrizing it more and more, finishing at $t=1$ in a ball of the same measure. The crucial fact is that during the evolution the first eigenvalue $\lambda_1(\Omega_t)$ decreases, as a consequence of the fact that $\Omega_t$ is more and more symmetric.

\proof[Proof of horizontal convexity]
We have to show that, given a domain $\Omega\in\A(D)$ with $|\Omega|\le m$, we can construct domains $\Omega_t\in\A(D)$ with $|\Omega_t|\le m$ such that
\be\label{evercon}
\lambda_1(\Omega_t)=\lambda_1(\Omega),\qquad\lambda_1(\Omega_t)=t\lambda_2(B)+(1-t)\lambda_2(\Omega),\qquad\forall t\in[0,1]
\ee
where, again, $B$ denotes the ball of measure $m$. If the continuous Steiner symmetrization $\Omega_t$ was $\gamma$-continuous, the proof could be easily achieved because the path
$$x(t)=\lambda_1(\Omega_t),\qquad y(y)=\lambda_2(\Omega_t)$$
would join continuously the point $\big(\lambda_1(\Omega),\lambda_2(\Omega)\big)$ to the point $\big(\lambda_1(B),\lambda_2(B)\big)$ with $x(t)$ descreasing, and this, together with the conicity property (v) of the set $\EE$ would allow us to conclude that for every $t\in[0,1]$ there is a set $\Omega_t$ verifying the property \eqref{evercon}. Actually, the continuous Steiner symmetrization path $\Omega_t$ provides a mapping $t\mapsto\lambda_2(\Omega_t)$ which is only lower semicontinuous on the left and upper semicontinuous on the right, but this is still enough to achieve the proof, using again the conicity property (v) of the set $\EE$.
\endproof

Summarizing, we have obtained the following existence result

\begin{teor}\label{texnoconv}
If $D$ is large enough, and $\Phi$ is lower semicontinuous and verifies the coercivity condition \eqref{ecoerc}, then the shape optimization problem \eqref{etwo} admits a solution.
\end{teor}

In the proof, several properties of the first two eigenvalues of the Laplace operator have been used; we do not know if more general problems of the form
$$\min\Big\{\Phi\big(\lambda_{i_1}(\Omega),\dots,\lambda_{i_k}(\Omega)\big)\ :\ \Omega\in\A(D),\ |\Omega|\le m\Big\}$$
still admit an optimal domain solution, without monotonicity assumptions on the function $\Phi$, for different choices of the indices $i_1,\dots,i_k$.

\section{Optimal partition problems}\label{spart}

\subsection{Existence of optimal partitions}\label{ssoptpar}

In this section we consider the minimization problem for shape cost functionals of the form
$$F(\Omega_1,\dots\Omega_N)$$
where $N$ is a fixed integer, the unknown domains $\Omega_i\in\A(D)$, the cost is a map $F:\A(D)^N\to[0,+\infty]$ and the domains $\Omega_i$ fulfill the condition
$$\Omega_i\cap\Omega_j=\emptyset\qquad\hbox{for }i\ne j.$$
A family $\{\Omega_1,\dots,\Omega_N\}$ of pairwise disjoint subsets of $D$ will be called a {\it partition}. The minimization problem we consider is then
\be\label{epart}
\min\big\{F(\Omega_1,\dots\Omega_N)\ :\ \Omega_j\in\A(D),\ \Omega_i\cap\Omega_j=\emptyset\hbox{ for }i\ne j\big\}.
\ee

As in the case of a single domain $\Omega$ we may consider several interesting cases of partition cost functionals $F$:
\begin{enumerate}
\item{\it integral costs} given by
$$F(\Omega_1,\dots\Omega_N)=\int_Dj\big(x,\RR_{\Omega_1}(f_1),\dots,\RR_{\Omega_N}(f_N)\big)\,dx$$
where $f_i\in L^2(D)$ are given and as usual the solutions $\RR_{\Omega_i}(f_i)$ of
$$-\Delta u=f_i\hbox{ on }\Omega_i\qquad u\in H^1_0(\Omega_i)$$
are intended as extended by zero outside $\Omega_i$
\item{\it spectral costs} given by
$$F(\Omega_1,\dots\Omega_N)=\Phi\big(\lambda(\Omega_1),\dots,\lambda(\Omega_N)\big)$$
where we denoted by $\lambda(\Omega)$ the spectrum of the Dirichlet Laplacian in $\Omega$.
\end{enumerate}

By the same methods used in Section \ref{ssnece} we may obtain some necessary conditions of optimality for the case of integral costs (i) above. Assuming that an optimal partition $(\Omega_1,\dots,\Omega_N)$ exists and is made of smooth domains we have
\be\label{eoptpart}
\left\{\begin{array}{ll}
u_iv_i\le0&\hbox{q.e. on }\Omega_i;\\
f_i(x)j_{s_i}(x,0,\dots,0)\ge0&\hbox{a.e. on }D\setminus\Omega_i;\\
(\partial u_i/\partial n)(\partial v_i/\partial n)=0
&\HH^{N-1}\hbox{-a.e. on }D\cap\big(\partial\Omega_i\setminus\bigcup_{j\ne i}\partial\Omega_j\big)
\end{array}\right.
\end{equation}
for all $i=1,\dots,N$, where we denoted by $u_i$ and $v_i$ the solutions of the state and adjoint state equation
$$-\Delta u=f_i\hbox{ on }\Omega_i\qquad u\in H^1_0(\Omega_i)$$
and
$$-\Delta v=j_{s_i}(x,u_1,\dots,u_N)\hbox{ on }\Omega_i\qquad v\in H^1_0(\Omega_i)$$
respectively.

In general, without extra conditions either on the admissible partitions or on the cost functional $F$, we could reproduce counterexamples to the existence of optimal partitions similar to the ones seen in Section \ref{ssnonex}, and optimal solutions only exist in a relaxed sense. For optimal partition problems the relaxed formulation will be considered later in Section \ref{ssrelpar}; here we focus our attention on some monotonicity assumptions on the cost $F$ that will imply the existence of {\it classical solutions} to problem \eqref{epart}.

\begin{defi}\label{dmonpar}
We say that $F:\A(D)^N\to[0,+\infty]$ is:
\begin{enumerate}
\item $\gamma$-lower semicontinuous if
$$F(\Omega_1,\dots,\Omega_N)\le\liminf_{n\to+\infty}F(\Omega_1^n,\dots,\Omega_N^n)$$
whenever $\Omega_i^n\to\Omega_i$ in the $\gamma$-convergence, for $i=1,\dots,N$;
\item $w\gamma$-lower semicontinuous if
$$F(\Omega_1,\dots,\Omega_N)\le\liminf_{n\to+\infty}F(\Omega_1^n,\dots,\Omega_N^n)$$
whenever $\Omega_i^n\to\Omega_i$ in the $w\gamma$-convergence, for $i=1,\dots,N$;
\item monotonically decreasing in the sense of the set inclusion if for all $(A_1,\dots,A_N)$, $(B_1,\dots,B_N)\in\A(D)^N$ such that $A_i\subset B_i$ for all $i=1,\dots,N$ in the sense of capacity, i.e. $\Cp(A_i\setminus B_i)=0$, then
$$F(B_1,\dots,B_N)\le F(A_1,\dots,A_N).$$
\end{enumerate}
\end{defi}

The following existence result for optimal partition problems with monotonically decreasing shape cost functional has been obtained in \cite{bbh98}.

\begin{teor}\label{texpar}
Let $F:\A(D)^N\to[0,+\infty]$ be a shape cost functional which is $w\gamma$-lower semicontinuous. Then the optimal partition problem \eqref{epart} admits a classical solution. This happens for instance if $F$ is $\gamma$-lower semicontinuous and monotonically decreasing in the sense of the set inclusion.
\end{teor}

For instance Theorem \ref{texpar} applies to the case
$$F(\Omega_1,\dots,\Omega_N)=\Phi\big(\lambda_{k_1}(\Omega_1),\dots,\lambda_{k_N}(\Omega_N)\big)$$
where $k_i\ge1$ are integers and $\Phi$ is increasing in each variable, i.e.
$$\Phi(s_1,\dots,s_N)\le\Phi(t_1,\dots,t_N)\qquad\hbox{whenever }s_i\le t_i,\ i=1,\dots,N.$$
Note that no measure constraints on the sets $\Omega_i$ are imposed; however, thanks to the $w\gamma$-lower semicontinuity result of Proposition \ref{plebes}, if $F:\A(D)^N\to[0,+\infty]$ is $w\gamma$-lower semicontinuous, the measure constrained minimization problem
$$\min\big\{F(\Omega_1,\dots\Omega_N)\ :\ \Omega_j\in\A(D),\ \Omega_i\cap\Omega_j=\emptyset\hbox{ for }i\ne j,\ |\Omega_i|\le m_i\big\}$$
still admits a classical solution.

A particularly interesting example is given by the shape cost functional
$$F(\Omega_1,\dots,\Omega_N)=\lambda_1(\Omega_1)+\dots+\lambda_1(\Omega_N).$$
In \cite{cali07,cali10} Caffarelli and Lin considered the equivalent variational formulation of the problem
$$\min\Big\{\int_D|\nabla u|^2\,dx\ :\ u\in H^1_0(D;\R^N),\ \int_Du_i^2\,dx=1\hbox{ for all }i,\ G(u)=0\Big\}$$
where
$$G(u)=\sum_{i\ne j}u_i^2u_j^2,$$
and showed the regularity of the free boundary of the optimal partition, up to a singular set of Hausdorff dimension less than or equal to $d-2$. They also formulated a very interesting conjecture about the behaviour of the optimal partitions when the number of parts $N$ tends to $+\infty$:

\medskip{\it in the two-dimensional case $d=2$ the optimal partitions $\{\Omega_i^{opt}\}$ tend, as $N\to+\infty$, to be made by a regular exagonal tiling; in particular, for the optimal value $P_N=\sum_{i=1}^N\lambda_1(\Omega_i^{opt})$ the estimate
$$N^{-2}P_N=\frac{\lambda_1(H)}{|D|}+o(1)$$
holds, where $H$ is a regular exagon of area equal to $1$.}\medskip

The numerical computations made in \cite{bbo09} confirm the conjecture, as Figure \ref{fbbo} shows.
\begin{figure}[h]
\centerline{
\includegraphics[width=4.3cm]{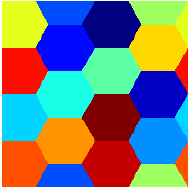}
\includegraphics[width=4.29cm]{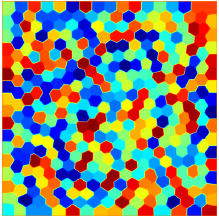}
\includegraphics[width=4.3cm]{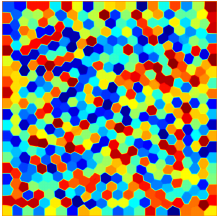}}
\caption{Optimal partitions for $\sum_{i=1}^N\lambda_1(\Omega_i)$ with $N=16$, $N=384$, $N=512$.}\label{fbbo}
\end{figure}

\subsection{Relaxed formulation of optimal partition problems}\label{ssrelpar}

For a general partition functional $F(\Omega_1,\dots,\Omega_N)$ which does not satisfy any monotonicity assumption we cannnot expect the existence of a classical solution to the optimization problem \eqref{epart}. Similarly to what happens in the case of a single domain, a relaxation procedure is needed to describe the behaviour of minimizing sequences, and this will be done through the use of the capacitary measures described in Sections \ref{ssmeas} and \ref{ssrelax}.

In order to characterize the expression of the relaxed problem associated to \eqref{epart} we consider, for every sequence $\{\Omega_1^n,\dots,\Omega_N^n\}$ of pairwise disjoint quasi-open subsets of $D$ the associated measures $\mu_i^n=\infty_{D\setminus\Omega_i^n}$ ($i=1,\dots,N$) of the class $\M_0(D)$ introduced in Section \ref{ssmeas}. Since $\M_0(D)$ is compact with respect to the $\gamma$-convergence, up to a subsequence, there exist $N$ measures $\mu_i\in\M_0(D)$ ($i=1,\dots,N$) such that $\mu_i$ is the $\gamma$-limit of $\mu_i^n$.

The fact that the original $\Omega_i^n$ were pairwise disjoint will imply some conditions on the limit measures $\mu_i$, that are not ``independent'', and characterize the relaxed optimization problem associated to \eqref{epart}. For instance, it is not possible to obtain an $N$-tuple made by all the measures $\mu_i$ equal to the Lebesgue measure on $D$.

For every capacitary measure $\mu\in\M_0(D)$ we denote by $\Omega_\mu$ the set ``of finiteness'' of $\mu$, more precisely defined as
$$\Omega_\mu=\big\{\RR_\mu(1)>0\big\}.$$
The following result has been proved in \cite{buti02}.

\begin{teor}\label{tbuti}
An $N$-tuple $(\mu_1,\dots,\mu_N)$ of capacitary measures is made by the $\gamma$-limit of pairwise disjoint quasi-open sets $\{\Omega_1^n,\dots,\Omega_N^n\}$ if and only if it satisfies the following property: 
$$\Cp(\Omega_{\mu_i}\cap\Omega_{\mu_j})=0\qquad\forall i\ne j.$$
\end{teor}

For instance, if
$$F(\Omega_1,\dots\Omega_N)=\int_Dj\big(x,\RR_{\Omega_1}(f_1),\dots,\RR_{\Omega_N}(f_N)\big)\,dx,$$
then the relaxed problem associated to
$$\min\big\{F(\Omega_1,\dots\Omega_N)\ :\ \Omega_i\in\A(D),\ \Omega_i\cap\Omega_j=\emptyset\hbox{ for }i\ne j\big\}$$
is given by
$$\min\big\{F(\mu_1,\dots\mu_N)\ :\ \mu_i\in\M_0(D),\ \Cp(\Omega_{\mu_i}\cap\Omega_{\mu_j})=0\hbox{ for }i\ne j\big\}$$
where
$$F(\mu_1,\dots\mu_N)=\int_Dj\big(x,\RR_{\mu_1}(f_1),\dots,\RR_{\mu_N}(f_N)\big)\,dx.$$

\section{Further problems}\label{sfurth}

\subsection{Cheeger-type problems}\label{sschee}

We consider here minimization problems that we may call of Cheeger type, for the similarity with the well-known Cheeger problem
$$\min\Big\{\frac{\per(\Omega)}{|\Omega|}\ :\ \Omega\subset D\Big\}.$$
It is possible to prove that for every bounded domain $D$ there exists an optimal Cheeger set $\Omega_{opt}$ and (see \cite{ac07,ccn07}) this set is unique and convex whenever $D$ is convex. Moreover, in this case the boundary $\partial\Omega_{opt}$ does not contain the points of $\partial D$ with too large mean curvature; more precisely, $\partial\Omega_{opt}$ coincides with $\partial D$ if and only if
$$\|H\|_{L^\infty(\partial D)}\le\frac{\lambda(D)}{d-1}$$
where $H(x)$ is the mean curvature of $\partial D$ at $x$ and $\lambda(D)$ is the minimal value of the Cheeger minimization problem above. Note that the two quantities $\per(\Omega)$ and $|\Omega|$ scale in a different way and, due to the isoperimetric inequality
$$\per(\Omega)|\Omega|^{(1-d)/d}\ge\per(B)|B|^{(1-d)/d}=d\omega_d^{1/d}$$
(where $\omega_d$ denotes the Lebesgue measure of the unit ball in $\R^d$), we have
$$\lim_{|\Omega|\to0}\frac{\per(\Omega)}{|\Omega|}=+\infty.$$
The same analysis occurs for the cost
$$|\Omega|^{-\alpha}\per(\Omega)\qquad\hbox{whenever }\alpha>\frac{d-1}{d}.$$

In this section we consider rescaled minimization problems of the form
\be\label{echee}
\min\big\{M(\Omega)J(\Omega)\ :\ \Omega\in\A(D)\big\}
\ee
where the mappings $M$ and $J$ fulfill some rather general assumptions related to the variational $\gamma$-convergence on the family $\A(D)$ of quasi-open subsets of $D$. These problems have been studied in \cite{bw10} to which we refer for all the details. On the mappings $M$ and $J$, from $\A(D)$ into $[0,+\infty]$ we assume
\begin{enumerate}
\item $M$ and $J$ are nonnegative, and $J(D)>0$;
\item $J$ is $\gamma$-l.s.c. and decreasing with respect to the set inclusion;
\item $M$ is $w\gamma$-lower semicontinuous;
\item the Cheeger scaling condition is verified:
$$\lim_{M(\Omega)\to0}M(\Omega)J(\Omega)=+\infty.$$
\end{enumerate}

With this last assumption we may define the cost $M(\Omega)J(\Omega)=+\infty$ whenever $M(\Omega)=0$ and obtain the following existence result.

\begin{teor}\label{tchee}
Under the assumptions above the minimum problem \eqref{echee} admits a solution $\Omega_{opt}$ and $M(\Omega_{opt})>0$.
\end{teor}

\begin{proof}
If $(\Omega_n)$ is a minimizing sequence, by the compactness of the $w\gamma$-convergence (see Section \ref{ssweakg}) we may assume that $\Omega_n$ tend in $w\gamma$ to some $\Omega\in\A(D)$, with $M(\Omega)>0$ by the assumptions above. By the properties of $w\gamma$-convergence and by Proposition \ref{pmonot} the functional $J$ is $w\gamma$-lower semicontinuous, as well as the product $MJ$, which allows to conclude that $\Omega$ is an optimal domain for the minimum problem \eqref{echee}.
\end{proof}

\begin{exam}
The following cases fulfill the assumptions of Theorem \ref{tchee}.
\begin{itemize}
\item Let $J(\Omega)=\Phi\big(\lambda(\Omega)\big)$ where $\lambda(\Omega)$ denotes the spectrum of the Dirichlet Laplacian in $\Omega$ and $\Phi$ is a lower semicontinuous function increasing in each variable and positively $p$-homogeneous. Taking $M(\Omega)=|\Omega|^\alpha$ we have that the assumptions of Theorem \ref{tchee} are fulfilled whenever $\alpha<2p/d$, so that in this case the minimization problem
$$\min\big\{|\Omega|^\alpha\Phi\big(\lambda(\Omega)\big)\ :\ \Omega\in\A(D)\big\}$$
admits a solution.

\item The case $J(\Omega)$ as above with $M(\Omega)=|\per(\Omega)|^\alpha$ does not fall into the framework of Theorem \ref{tchee} because the shape functional $\per(\Omega)$ is not $w\gamma$-lower semicontinuous (see Section \ref{ssperim}). However, using the arguments of Section \ref{ssperim}, we obtain that the minimization problem
$$\min\big\{|\per(\Omega)|^\alpha\Phi\big(\lambda(\Omega)\big)\ :\ \Omega\in\A(D)\big\}$$
admits a solution whenever $\alpha<2p/(d-1)$.

\item Consider the integral functional
$$I(\Omega)=\int_Da(x)|u_\Omega(x)|^p\,dx$$
where $u_\Omega$ is the solution of the Dirichlet problem in $\Omega$
$$-\Delta u=f\hbox{ in }\Omega,\qquad u\in H^1_0(\Omega),$$
extended by zero on $D\setminus\Omega$. Assume that $a\in L^q$ is nonnegative, with $q\ge p/(p-1)$, $f\in L^2$ is nonnegative, and $0<p<2d/(d-2)$. Then the shape functional
$$J(\Omega)=1/I(\Omega)$$
fulfills the assumptions of Theorem \ref{tchee} and, taking $M(\Omega)=|\Omega|^\alpha$ we have that the minimization problem
$$\min\Big\{\frac{|\Omega|^\alpha}{I(\Omega)}\ :\ \Omega\in\A(D)\Big\}$$
admits a solution whenever $\alpha<(2p+d)/d$. A similar analysis holds for the shape functional
$$J(\Omega)=\int_Da(x)|u_\Omega(x)|^{-p}\,dx.$$

\item The case of the minimization problem
$$\min\Big\{\frac{|\per(\Omega)|^\alpha}{I(\Omega)}\ :\ \Omega\in\A(D)\Big\}$$
where $I(\Omega)$ is as above, does not fall into the framework of Theorem \ref{tchee} but the arguments of Section \ref{ssperim} apply, and we can conclude that an optimal domain exists, provided that $\alpha<(2p+d)/(d-1)$.
\end{itemize}
\end{exam}

As it happens in the Cheeger problem, due to the scaling condition we assumed, the optimal domains $\Omega_{opt}$ for the examples above have to touch the boundary $\partial D$. Here below we describe some necessary conditions of optimality, obtained in \cite{bw10} for the minimization problem
$$\min\Big\{|\Omega|^\alpha\Big(\int_Du_\Omega\,dx\Big)^{-1}\ :\ \Omega\in\A(D)\Big\}$$
where $-\Delta u_\Omega=1$ in $\Omega$, with Dirichlet boundary conditions.

\begin{teor}\label{tdomainvar}
Assuming that $\Omega_{opt}$ is smooth, and that $\alpha<(2+d)/d$, the following necessary conditions of optimality hold:
$$\left\{
\begin{array}{ll}
|\nabla u_{\Omega_{opt}}|^2=\alpha K&\hbox{ on }\partial\Omega_{opt}\cap D\\
|\nabla u_{\Omega_{opt}}|^2\ge\alpha K&\hbox{ on }\partial\Omega_{opt}\cap\partial D
\end{array}\right.
\qquad\hbox{ where }
K=\frac{1}{|\Omega_{opt}|}\int_D u_{\Omega_{opt}}\,dx.$$
\end{teor}

\begin{rema}Assume that $\partial D$ has isolated conical points, but is smooth otherwise. The gradient of the solution $u_{\Omega_{opt}}$ then
vanishes in the conical points and a pointwise decay estimate for the gradient holds. Theorem \ref{tdomainvar} then shows, that the optimal domain $\Omega_{opt}$ cannot fill the entire set $D$.
\end{rema}

In a similar way we can obtain necessary conditions of optimality for the minimum problem
$$\min\big\{|\Omega|^\alpha\lambda_k(\Omega)\ :\ \Omega\in\A(D)\Big\}$$
where $\alpha<2/d$. If we denote by $u$ a normalized $k$-th eigenfunction of $\Omega_{opt}$ we obtain (see \cite{bw10}):
$$\left\{
\begin{array}{ll}
|\nabla u|^2=\alpha H\hbox{ on }\partial\Omega_{opt}\cap D;\\
|\nabla u|^2\ge\alpha H\hbox{ on }\partial\Omega_{opt}\cap\partial D.
\end{array}\right.
\qquad\hbox{ where }
H=\frac{\lambda_k(\Omega_{opt})}{|\Omega_{opt}|}.$$

We conclude this section by pointing out some shape optimization problems for which the existence of a solution is still unavailable. For a fixed $k\ge1$ consider the minimization problem
$$\min\big\{\lambda_k^\alpha(\Omega)\int_Du_\Omega\,dx\ :\ \Omega\in\A(D)\big\}$$
with $\alpha>1+d/2$. By the results of \cite{kj78a,kj78b} we have
$$\lambda_1^{1+d/2}(\Omega)\int_Du_\Omega\,dx\ge\lambda_1^{1+d/2}(B)\int_Du_B\,dx$$
for any ball $B$, so that the scaling condition
$$\lim_{M(\Omega)\to0}M(\Omega)J(\Omega)=+\infty$$
is fulfilled whenever $\alpha>1+d/2$, taking
$$J(\Omega)=\lambda_k^\alpha(\Omega)\quad\hbox{and}\quad M(\Omega)=\int_Du_\Omega\,dx.$$
However, the $w\gamma$-lower semicontinuity condition fails for $M(\Omega)$, as it can be easily seen by a sequence of finely perforated domains like the Cioranescu and Murat ones (see Section \ref{ssmeas}), and so the existence Theorem \ref{tchee} cannot be applied. It would be interesting to prove (or disprove) that an optimal domain $\Omega_{opt}$ for the problem above exists.

\subsection{Spectral flows}\label{ssflow}

In \cite{dg93} De Giorgi introduced a very general theory to study evolution problems with an underlying variational structure. The theory was called {\it minimizing movement theory} and its framework is so flexible to be applied both to quasistatic evolutions as well as to gradient flows, under rather mild assumptions. Here we limit ourselves to recall the part of the scheme which is interesting for our purposes, referring to \cite{ags05} for further details.

Let $(X,d)$ be a complete metric space, let $u_0\in X$ be an initial condition, and let $F:X\to]-\infty,+\infty]$ be a functional defined on $X$. For every fixed $\eps>0$ the implicit Euler scheme of time step $\eps$ and initial condition $u_0$ consists in constructing a function $u_\eps(t)=w([t/\eps])$, where $[\cdot]$ stands for the integer part function, in the following way:
$$w(0)=u_0,\qquad w(n+1)\in\argmin\Big\{F(v)+\frac{d^2(v,w(n))}{2\eps}\Big\}.$$

\begin{defi}\label{dmovem}
If $\sigma$ is another topology on $X$, we say that $u:[0,T]\to X$ is a generalized minimizing movement (or a variational flow) associated to $F$ and $\sigma$, with initial condition $u_0$, and we write $u\in GMM(F,\sigma,u_0)$, if there exist a sequence $\eps_n\to0$ such that
$$u_{\eps_n}(t)\to u(t)\hbox{ in }\sigma\qquad\forall t\in[0,T].$$ 
\end{defi}

The simplest situation occurs when $\sigma$ is the topology of the metric $d$; in this case the following existence result for variational flows holds (see Proposition 2.2.3 of \cite{ags05}).

\begin{teor}\label{texmm}
If $F$ is $d$-lower semicontinuous and $d$-coercive, in the sense that its sublevels $\{F\le k\}$ are $d$-compact in $X$, then for every initial condition $u_0\in X$ there exists $u\in GMM(F,d,u_0)$. Moreover, $u$ belongs to the space $AC^2\big((0,T);X\big)$ of mappings such that, for a suitable $m\in L^2(0,T)$
$$d\big(u(s),u(t)\big)\le\int_s^t m(r)\,dr\qquad\forall s,t\in[0,T].$$
\end{teor}

The result above applies straightforward, thanks to Proposition \ref{pmeas} (i), to the case $X=\M_0(D)$ endowed with the compact distance $d_\gamma$. By Theorem \ref{texmm} for every initial condition $\mu_0\in\M_0(D)$ there exists a $GMM(F,d_\gamma,\mu_0)$ flow $\mu(t)$ which is of class $AC^2\big((0,T);\M_0(D)\big)$.

There is a natural one-to-one map between the class of capacitary measures $\M_0(D)$ and the closed convex set of $L^2(D)$
$$K=\{w\in H^1_0(D)\ :\ w\ge0,\ 1+\Delta w\ge0\},$$
given by
$$\mu\mapsto\RR_\mu(1):=w_\mu,\quad\hbox{with inverse}\quad w\mapsto\mu_w=\frac{1+\Delta w}{w}\;.$$
Moreover, the metric structures on $\M_0(D)$ and $K$ are the same, since 
$$d_\gamma(\mu_1,\mu_2)= \|w_{\mu_1}-w_{\mu_2}\|_{L^2(D)}.$$
Therefore, every functional $F:\M_0(D)\to\overline\R$ can be identified with a functional $J:K\to\overline\R$ by
$$F(\mu)=J(w_\mu)\quad\hbox{or equivalently}\quad J(w)=F(\mu_w).$$
The variational flow for $F$ in $\M_0(D)$ can be then obtained through the variational flow for $J$ in the metric space $K$ endowed with the $L^2(D)$ distance, generated by the implicit Euler scheme
\begin{equation}\label{eimplw}
w_\eps^{n+1}\in\argmin\Big\{J(w)+\frac{1}{2\eps}\int_D|w-w_\eps^n|^2\,dx\Big\}.
\end{equation}
For instance, we may consider a functional $J$ of integral type, of the form
$$J(w)=\int_D j(x,w(x))\,dx,$$
where $j:D\times\R\to\R$ is a suitable integrand. In particular, we may take $j(x,w)=-w$ which leads to the energy of the system for the constant force $f\equiv1$, or  $j(x,w)=w$ which gives the torsional rigidity. If $j(x,s)$ is convex and lower semicontinuous in the second variable the variational flow for capacitary measures is reduced to the gradient flow of a convex lower semicontinuous map on $L^2(D)$.

An interesting question is the following: if we start from an initial condition which is a quasi-open set, i.e. $\mu_0=\infty_{D\setminus\Omega_0}$, in which cases the flow remains in the family of quasi-open sets? It can be shown that in general this does not happen, at least at the discrete level. This kind of phenomenon was numerically observed in the framework of quasi-static debonding membranes \cite{bbl08}.

We may consider, instead of capacitary measures, flows of shapes by endowing the class $\A(D)$ of quasi-open subsets of $D$ with a suitable distance. There are no ``standard'' distances on $\A(D)$ and several choices can be made. A first possibility is to consider the Lebesgue measure of the symmetric difference set
$$d_{char}(\Omega_1,\Omega_2)=|\Omega_1\symd\Omega_2|.$$
Since two quasi-open sets may differ for a negligible set (as for instance in $\R^2$ a disk and a disk minus a segment), this is not a proper metric in $\A(D)$, so that one should consider equivalence classes in the family of shapes.

The distance $d_{char}$ is not compact on $\A(D)$; nevertheless, for $\gamma$-lower semicontinuous functionals which are monotone decreasing for the set inclusion, we have the following result.

\begin{teor}\label{bba07}
Let $F:\A(D)\to\overline\R$ be a $\gamma$-lower semicontinuous functional monotone decreasing for the set inclusion, and let $\Omega_0\in\A(D)$. There exists a GMM map $t\mapsto\Omega(t)$ associated to $F$ and to the distance $d_{char}$, with initial condition $\Omega_0$. Moreover, the flow $\Omega(t)$ is increasing for the set inclusion.
\end{teor}

For instance, spectral functionals of the form
$$F(\Omega)=\Phi\big(\lambda(\Omega)\big),$$
where $\lambda(\Omega)$ is the spectrum of the Dirichlet Laplacian in $\Omega$ and $\Phi:\R^k\to\overline\R$ is increasing in each variable and lower semicontinuous, fulfill the assumptions above.

Several questions about the flow $\Omega(t)$ arise.
\begin{itemize}
\item The flow $\Omega(t)$ is not obtained through Theorem \ref{texmm} so the functional $F$ could be discontinuous on the curve $\Omega(t)$; it would be interesting to analyze this issue and to identify some cases where this continuity occurs.

\item In $\R^2$, assume that $\Omega_0$ is simply connected. Prove or disprove that the GMM associated to $\lambda_1$ in the framework of Theorem \ref{bba07} consists only on simply connected open sets.

\item Assume that $\Omega_0$ is convex. Prove or disprove that a GMM associated to $\lambda_1$ in the framework of Theorem \ref{bba07} consists only on convex sets.

\item Assume that $\Omega_0$ is convex. Is it true that any minimizer of
$$\min_{\Omega_0\subset\Omega}\lambda_1(\Omega)+|\Omega|,$$
is convex?

\item Is it true that the GMM associated to $\lambda_1$ in the framework of Theorem \ref{bba07} converges, up to a suitable rescaling, to a ball?

\item Prove or disprove that the metric derivative of $\lambda_1$, computed at a bounded smooth set $\Omega$ is given by ($u_1$ denotes the normalized first eigenfunction)
$$|\lambda_1'(\Omega)|
=\limsup_{|\Omega_n\setminus\Omega|\to0,\ \Omega\subset\Omega_n}\frac{\lambda_1(\Omega)-\lambda_1(\Omega_n)}{|\Omega_n\setminus\Omega|}
=\max_{\partial\Omega}\Big|\frac{\partial u_1}{\partial n}\Big|^2.$$

\end{itemize}

Other possibilities for defining spectral flows are possible, as for instance the frameworks below.

\begin{itemize}\renewcommand{\labelitemi}{$\circ$}
\item Work in the class of sets with prescribed measure $\{\Omega\in\A(D)\ :\ |\Omega|=m\}$.

\item Introduce a penalization on the perimeter, that is work with the functional $F(\Omega)+\per(\Omega)$.

\item Work with the Hausdorff complementary distance $d_{H^c}$ in the family of open subsets of $D$:
$$d_{H^c}(\Omega_1,\Omega_2)=\max_{x\in\overline D}|d(x,\overline D\setminus\Omega_1)-d(x,\overline D\setminus\Omega_2)|.$$

\item Work in the class of convex open subsets $\Omega$ metrized by one of the distances:\\
- The Hausdorff distance;\\
- The $L^1$ distance $d_{char}$ of the characteristic functions;\\
- The $L^2$ distance of the oriented distance functions.
\end{itemize}

\subsection{Other kinds of boundary conditions}\label{ssneum}

All the results presented in the previous sections are concerned with the framework of {\it Dirichlet boundary conditions}. The main reason why this  framework allows to treat the various questions we have seen, is the monotonicity relation
$$f\ge0,\ \Omega_1\subset\Omega_2\quad\Rightarrow\quad\RR_{\Omega_1}(f)\le\RR_{\Omega_2}(f)\hbox{ and }\lambda(\Omega_2)\le\lambda(\Omega_1).$$
When working with different boundary conditions, the monotonicity properties above are no more true, which creates several difficulties if one wants to obtain similar kinds of results. In addition, while a function $u\in H^1_0(\Omega)$ is easily extended to $H^1_0(D)$ by putting $u=0$ outside $\Omega$, similar extension operators do not exist if the Dirichlet boundary condition is replaced by {\it Neumann} or {\it Robin} type conditions. Finally, in order to work with solutions of elliptic PDEs with Neumann or Robin boundary conditions, some mild regularity on the domains $\Omega$ has to be required, which could be lost when passing to the limit on sequences $(\Omega_n)$.

Nevertheless, some shape optimization results are known in these more difficult frameworks; we mention few of them together with some open problems, referring for instance to \cite{hen06} for more details.

\bigskip

$\bullet$ {\it Neumann boundary conditions.}
If we denote by $\mu(\Omega)$ the spectrum of the Neumann-Laplacian on the domain $\Omega$ (notice that $\mu_1(\Omega)=0$ for all $\Omega$), the minimization problem for $\mu_k(\Omega)$ under the volume constraint $|\Omega|=m$ has a trivial solution. Indeed, we have $\mu_k(\Omega)=0$ for any domain $\Omega$ which has at least $k$ connected components.

If we impose that the domains $\Omega$ must be convex and with a given diameter, then the infimum is not zero, but it is not achieved, as shown in \cite{pawe60}.

A more interesting problem is the {\it maximization} of $\mu_k(\Omega)$ under a volume constraint. It is known (see \cite{sz54,we56}) that for $\mu_2(\Omega)$ the maximum is attained for $\Omega$ a ball, which are the only maximizers. An analogous result for $k\ge3$ (i.e. existence of optimal domains and possibly their characterization) is not known. Similarly, the existence of a maximizing domain for $\mu_2(\Omega)$, with the constraint $|\Omega|=m$, is not known if we impose the constraint $\Omega\subset D$ with $D$ narrow enough to not contain a ball of measure $m$.

Some more issues on spectral optimization problems for the Neumann-Laplacian, together with related numerical computations, can be found in \cite{bbbs06}. One could also consider mixed boundary conditions: Dirichlet on one part and Neumann on the remaining part. We mention the paper \cite{couh99} where this problem has been studied, and \cite{ar95} where the behaviour of the spectrum of the Neumann Laplacian under boundary perturbations is analyzed. Finally, some results are available for minimization problems of functions of reciprocals of $\mu_k$ under volume constraint (see \cite{hen06}) as:
$$\frac{1}{\mu_2(\Omega)}+\frac{1}{\mu_3(\Omega)}\;,\qquad\sum_{k\ge2}\frac{1}{\mu_k(\Omega)}\;,\qquad\sum_{k\ge2}\frac{1}{\mu_k^2(\Omega)}\;.$$

\bigskip

$\bullet$ {\it Robin boundary conditions.}
If $\Omega$ is a sufficiently regular open subset of $\R^d$ the Robin eigenvalue problem for the Laplace operator is given by
$$-\Delta u=\eta u\hbox{ in }\Omega,\qquad\frac{\partial u}{\partial n}+\beta u=0\hbox{ on }\partial\Omega,$$
where $\beta>0$ is a given real number. It is known (see for instance \cite{da06,buda10,bo86,bo88}) that, denoting by $\eta_1(\Omega)$ the smallest Robin eigenvalue, the minimum of $\eta_1(\Omega)$ among the class of bounded Lipschitz sets $\Omega$ with prescribed volume is achieved on balls, which are also the unique minimizers.

However, the problem of setting spectral optimization problems with Robin boundary conditions for general nonsmooth domains, faces some difficulties dues to the presence of boundary terms. In \cite{bugi10} a new approach has been proposed, which, using the features of $SBV$ functions, introduced by De Giorgi for treating free discontinuity problems (see \cite{afp01}), allows to define $\eta_1(\Omega)$ also for nonsmooth domains. Denoting by $SBV^{1/2}(\R^d)$ the space
$$SBV^{1/2}(\R^d)=\big\{u:\R^d\to[0,+\infty[\hbox{ measurable and } u^2\in SBV(\R^d)\big\}$$
we may define $\eta_1(\Omega)$ for a general open set $\Omega$ through the Rayleigh type formula
$$\eta_1(\Omega)=\min\Big\{
\int_{\R^d}|\nabla u|^2\,dx+\beta\int_{J_u}\big[(u^+)^2+(u^-)^2\big]\,d\HH^{d-1}\ :\ \int_{\R^d}u^2\,dx=1\Big\}$$
where the minimum above is taken over the class of functions $u\in SBV^{1/2}(\R^d)$ with $|\spt u\setminus\Omega|=\HH^{d-1}(J_u\setminus\partial\Omega)=0$, and $\nabla u$ is the approximate gradient of $u$, $J_u$ is the set of discontinuity points of $u$, $u^\pm$ are the traces of $u$ on $J_u$ from the two sides. In this way in \cite{bugi10} it is proved that for every open set $\Omega$ the isoperimetric inequality
$$\eta_1(\Omega)\ge\eta_1(B)$$
holds, where $B$ is a ball with $|B|=|\Omega|$. Moreover, equality holds if and only if $\Omega$ is a ball up to a negligible set.

However, all the questions of general spectral optimization problems still remain open, as:
\begin{itemize}
\item[-] existence results under the constraint $\Omega\subset D$;
\item[-] treating more general spectral costs as $\Phi\big(\eta(\Omega)\big)$;
\item[-] regularity of optimal domains.
\end{itemize}


\end{document}